\documentclass[final,subeqn]{siamltex}

\usepackage[algo2e,boxed]{algorithm2e}
\usepackage{amsfonts}
\usepackage{bm}
\usepackage{enumerate}
\usepackage{graphicx}
\usepackage{mathrsfs}
\usepackage{pstricks}
\usepackage{subfigure}
\usepackage{tikz-cd}
\usepackage{url}

\newcommand{\Dim}{{\scriptsize \textsf{D}}}

\newcommand{\Zero}{\hat{0}}
\newcommand{\One}{\hat{1}}
\newcommand{\Int}{\mathrm{int}}

\title{
  Fluid Modeling and Boolean Algebra 
  for Arbitrarily Complex Topology
  in Two Dimensions}

\author{Qinghai Zhang\thanks{School of Mathematical Sciences,
  Zhejiang University,
  38 Zheda Road,
  Hangzhou, Zhejiang Province, 310027 China
  ({\tt qinghai@zju.edu.cn}).}
  \and
  Zhixuan Li\thanks{
  School of Mathematical Sciences,
  Zhejiang University,
  38 Zheda Road,
  Hangzhou, Zhejiang Province, 310027 China
  ({\tt zihinlai@163.com}).}
}

\begin{document}

\maketitle


\begin{abstract}
  We propose a mathematical model
 for fluids in multiphase flows
 in order to establish a solid theoretical foundation
 for the study of their complex topology,
 large geometric deformations,
 and topological changes such as merging. 
Our modeling space consists of
 regular open semianalytic sets with bounded boundaries,
 and is further equipped
 with constructive and algebraic definitions
 of Boolean operations.
Major distinguishing features of our model include 
 (a) topological information of fluids
 such as Betti numbers can be easily
 extracted in constant time, 
 (b) topological changes of fluids
 are captured by non-manifold points
 on fluid boundaries,
 (c) Boolean operations on fluids 
 correctly handle all degenerate cases
 and apply to arbitrarily complex topologies,
 yet they are simple and efficient
 in that they only involve
 determining the relative position of a point 
 to a Jordan curve
 and intersecting a number of curve segments.
Although the main targeting field
 is multiphase flows,
 our theory and algorithms
 may also be useful for related fields
 such as solid modeling, computational geometry,
 computer graphics, 
 and geographic information system.

 




\end{abstract}



\section{Introduction}
\label{sec:intro}

Physically meaningful regions
 in the sense of homogeneous continua
 are ubiquitous,
 and their modeling is of fundamental significance
 in innumerable applications of science and engineering.
Traditionally,
 the modeling of two- and three-dimensional
 physical regions
 is the main subject of a mature research field
 called solid modeling
 \cite{requicha83:_solid,shapiro02:_solid}.
In comparison,
 modeling of fluids have always been avoided
 in the field of multiphase flows.
However, rapid advancements in the science of multiphase flows
 have been calling for such a model 
 so that complex phenomena
 such as those involving topological changes of fluids
 can be studied rigorously.

In this paper, we aim to answer this need
 by introducing the notion of \emph{fluid modeling}
 in multiphase flows,
 analogous to solid modeling
 in computer-aided design (CAD).
We propose a topological space 
 for fluid modeling
 and further equip this space
 with natural algebraic structures
 in order to extract essential topological information
 and to perform simple and efficient
 Boolean operations.

In Sections \ref{sec:solid-modeling},
 \ref{sec:fluid-model-interf},
 and \ref{sec:boolean-operations},
 we motivate different aspects of fluid modeling
 and review previous efforts
 and relevant results.
We then list in Section \ref{sec:motiv-contr-this}
 a number of questions 
 as the more detailed targets of this work.

\subsection{Solid modeling}
\label{sec:solid-modeling}

What distinguishes solid modeling from
 similar disciplines such as computer graphics
 is its emphasis on \emph{physical fidelity},
 as evident in
 its underlying mathematical and computational principles.
This emphasis is natural:
 driven by the design, analysis, and manufacture
 of engineering systems,
 solid modeling must support
 the representation, visualization,
 exchange, interrogation, and creation
 of physical objects in CAD.

One common approach of solid modeling
 relies on point-set topology.
The classical modeling space
 proposed by Requicha and colleagues
 \cite{requicha77:_mathem_model_rigid_solid_objec,requicha78:_mathem_found_const_solid_geomet,requicha80:_repres_for_rigid_solid}
 consists of \emph{r-sets},
 which are bounded, closed, regular semianalytic sets in Euclidean spaces.
The regularity condition captures
 in solid continua the absence of
 low-dimensional features such as isolated gaps and points,
 and the semianalytic condition
 postulates that the boundary of a solid be locally well behaved;
 see Section \ref{sec:regularSets} for more details.

The other common approach in solid modeling
 is \emph{combinatorial}, in the sense of
 \emph{cell complexes} in algebraic topology
 \cite{munkres84:_elemen} \cite{saveliev16:_topol_illus}.
Complex objects are viewed in terms of
 primitive building blocks called \emph{cells},
 thus it is not the constituting cells 
 but their combinational informations
 that describe the physical object.
Take simplicial complexes for example,
 a \emph{$k$-cell} is a $k$-simplex
 in the Euclidean space $\mathbb{R}^k$,
 and many cells of different dimensions 
 are glued together
 to form an $n$-dimensional simplicial complex
 by requiring 
 that any adjacent pair of $k$-cells
 be attached to each other along a $(k-1)$-cell
 for each $k=1,\ldots, n$.
The adjacency of $k$-cells is
 encoded in the \emph{$k$th boundary operator} $\partial_k$,
 which maps each $k$-cell
 to an element in the \emph{$(k-1)$-chain} $C_{k-1}$,
 a group of formal sums of $(k-1)$-cells.
If we concatenate the chain groups
 with the boundary operators,
 we obtain a \emph{chain complex}, 
 \begin{equation}
   \label{eq:chainComplex}
   \begin{tikzcd}[column sep=2em]
     C_n \ar{r}{\partial_{n}}
     & \cdots \ar{r}{\partial_{k+1}}
     & C_k \ar{r}{\partial_{k}}
     & \cdots \ar{r}{\partial_{1}}
     & C_0
   \end{tikzcd}, 
 \end{equation}
 where each boundary operator is a group homomorphism.
This chain complex is all we need for
 mathematical modeling and computer representation
 of any $n$-dimensional solid!
As a prominent advantage, 
 key topological quantities,
 such as 
 the number of connected components
 and the number of holes,
 can be systematically computed from the chain complex.
The cost of this computation, however,
 can be substantial \cite{kaczynski04:_comput_homol}.

Thanks to the fact that
 any r-set can be represented by a simplicial complex
 as accurately as one wishes,
the point-set approach and the combinatorial approach
 are seamlessly consistent 
 \cite{requicha80:_repres_for_rigid_solid}.
Thus we can use these two models interchangeably,
 at least theoretically.
This consistency also forces
 an $n$-dimensional r-set 
 to be closed; 
 otherwise a boundary operator
 in (\ref{eq:chainComplex})
 may have a range outside of the chain complex.

\subsection{Fluid modeling and interface tracking (IT) in multiphase flows}
\label{sec:fluid-model-interf}

In dramatic comparison to the aforementioned 
 research on solid modeling,
 efforts on geometric modeling of fluids are rare:
 mathematical models and computer algorithms
 have been deliberately designed such that
 \emph{geometric} modeling of fluids is avoided
 in numerically simulating multiphase flows. 
In the volume-of-fluid (VOF) method \cite{Hirt.Nichols_1981_volume}, 
 a deforming fluid phase $\mathsf{M}$
 is represented by a \emph{color function} $f(\mathbf{x},t)$,
 of which the value is 1 if there is $\mathsf{M}$ at $({\bf x},t)$
 and 0 otherwise;
 then the fluid phase at time $t$ is represented as a moving point set
 $\mathcal{M}(t):=\{\mathbf{x}: f(\mathbf{x},t)=1\}$.
Either the scalar conservation law
 \begin{equation}
   \label{eq:SCL}
   \frac{\partial f}{\partial t}+ \nabla\cdot(f \mathbf{u}) = 0
 \end{equation}
 or the advection equation
 \begin{equation}
   \label{eq:advection}
   \frac{\partial f}{\partial t}+ \mathbf{u}\cdot\nabla f = 0 
 \end{equation}
 is solved to recover the boundary of ${\cal M}$
 at subsequent time instants.
In the level-set method \cite{osher88:_front_propag_curvat_speed},
 the boundary of a fluid phase is represented
 as the zero isocontour of a signed distance function $\phi$,
 and once again the region of the fluid phase
 is recovered by numerically solving
 either (\ref{eq:SCL}) or (\ref{eq:advection}) on $\phi$.
In the front tracking method
 \cite{tryggvason01:_front_track_method_comput_multip_flow},
 the boundary of a fluid phase
 is represented by connected Lagrangian markers;
 tracking the fluid phase is then reduced
 to tracking these markers
 via numerically solving ordinary differential equations.
In all of these IT methods,
 geometric problems in deforming fluids with sharp interfaces 
 are converted to
 numerically solving differential equations.
With topological information discarded,
 this conversion largely reduces
 the complexity of IT
 both theoretically and computationally; 
 this is a main reason
 for successes of the aforementioned IT methods.
During the past forty years,
 these IT methods have been extremely valuable
 in studying multiphase flows.

As the science of multiphase flows moves towards
 more and more complex phenomena,
 higher and higher expectations are imposed on IT.
First, 
 the wider and wider spectrum
 of relevant time scales and length scales
 in mainstream problems
 demands that IT methods
 be more and more accurate and efficient.
Second,
 the tight coupling of interface to ambient fluids
 necessitates accurate estimation of derived geometric quantities
 such as curvature
 and unit normal vectors.
Third,
 topological changes of a fluid phase
 such as merging
 exhibit distinct behaviors
 for different regimes of the Weber number
 and other impact parameters
 \cite{qian97:_regim},
 hence it is not enough to handle topological changes
 solely from the interface locus and the velocity field.
For these problems, 
 an IT method should also take as its input
 a policy that describes how the interface
 shall evolve at the branching time and place of topological changes.

Despite their tremendous successes,
 current IT methods have a number of limitations
 in answering the aforementioned challenges of multiphase flows.
First,
 most methods are at best second-order accurate
 \cite{zhang13:VOFadvection}.
Second,
 the IT errors put an upper
 limit on the accuracy of estimating
 curvature and unit vectors.
It is shown in \cite{zhang17:_hfes} that,
 for a second-order method,
 its error of curvature estimation 
 is proportional to $\sqrt{\epsilon_{p}}$
 where $\epsilon_{p}$ denotes a norm of IT errors.
In other words, 
 the number of accurate digits one gets
 in curvature estimation is at best
 half of that in the IT results.
Third,
 the avoidance of geometric modeling of fluids
 renders it highly difficult
 to treat topological changes rigorously.
When merging and separation happen, 
 front-tracking methods have to resort
 to ``surgical'' operations
 that are short of theoretical justification.
VOF methods and level-set methods
 have no special procedures for topological changes;
 this is often advertised as an advantage.
However,
 by using this ``automatic'' treatment, 
 an application scientist has \emph{no control}
 over the evolution of an interface
 that undergoes topological changes:
 the evolution is determined
 not by the physics,
 but by particularities of numerical algorithms \cite{zhang13:VOFadvection,zhang18:_cubic_mars_method_fourt_sixth}.
Clearly, this disadvantage is a consequence of avoiding
 the geometric and topological modeling of fluids.
 

In this work,
 we aim to \emph{establish a theoretical foundation for fluid modeling}.
To prevent reinventing the wheel,
 we have tried to utilize the wealth of solid modeling,
 but found that none of the two main approaches
 in Section \ref{sec:solid-modeling}
 is adequate for fluid modeling.
Topology computing
 based on cell complexes involves much machinery,
 yet its efficiency may not be acceptable;
 nor are the r-sets suitable for fluid modeling.
First, the requirement of r-sets being closed
 is not amenable to numerical analysis of IT
 methods \cite{zhang13:VOFadvection}.
Second,
 topological changes create on the fluid boundary
 a special type of non-manifold points, 
 such as the $q$ points in Figure \ref{fig:YinSets},
 which cause non-uniqueness of boundary representations,
 c.f. Figure \ref{fig:boundaryDecompositions}.
This non-uniqueness degrades isomorphisms
 to homomorphisms
 in the association of algebraic structures
 with elements in the topological modeling space.
Third,
 r-sets can be modified to
 yield a new Boolean algebra
 whose implementation is much simpler and more efficient.


\subsection{Boolean operations}
\label{sec:boolean-operations}

The operations to be performed on the modeling space
 are an indispensable part of the modeling process;
 after all, a major purpose of modeling
 is to answer questions on the objects being modeled. 
Hence theoretically
 a modeling space is shaped
 by primary cases of queries.
Computationally,
 these operations 
 should be defined algebraically and constructively
 so that they furnish realizable algorithms
 that finish in reasonable time.

We are interested in Boolean operations
 on physically meaningful regions
 with arbitrarily complex topologies.
This interest follows naturally
 from the motivations in Section \ref{sec:fluid-model-interf}.
First,
 we have shown that algorithms
 for clipping splinegons with a linear polygon
 can improve the IT accuracy by
 many orders of magnitudes
 \cite{zhang18:_cubic_mars_method_fourt_sixth}.
Second, for coupling an IT method to
 an Eulerian main flow solver,
 regions occupied by the fluid inside
 fixed control volumes
 are needed to define averaged values
 and to construct stencils
 for approximating spatial operators
 with linear combinations of these averaged values.
Third,
 in handling topological changes,
 the emerging time and sites
 of non-manifold points
 need to be detected on the fluid boundary
 before we are able to decide how to evolve it.
This detecting problem 
 requires calculating intersections
 of multiple regions inside a single control volume.
 
Boolean operations on polygons
 are an active and intense research topic
 in many related fields
 such as computational geometry,
 computer graphics, CAD, 
 and geographic information system (GIS).
In particular,
 physically meaningful regions in GIS
 such as parks, roads, and lakes are represented by polygons
 and their Boolean operations are essential for extracting information
 and answering queries.
Consequently,
 there exist numerous papers on this topic;
 see, e.g., 
 \cite{sutherland74:_rentr,liang83,vatti92,greiner98:_effic,orourke98:_comput_geomet_in_c,liu07,berg08:_comput_geomet,simonson10:_indus,martinez13:_boolean} 
 and references therein.
However,
 many current algorithms are subject to strong restrictions
 on operand polygons such as convexity,
 simple-connectedness,
 and no self-intersections.
In addition,
 most algorithms fail for degenerate cases
 such as a vertex of a polygon being
 on the edge of the other polygon;
 these degenerate scenarios,
 nonetheless, are at the core
 of characterizing and treating topological changes.
Therefore, current Boolean algorithms
 are not suitable for fluid modeling.

As another main reason for their lack of applicability
 in multiphase flows,
 very few of current Boolean algorithms
 have a solid mathematical foundation,
 and those that do have other notable drawbacks.
For example,
 Boolean algorithms based on cell complexes
 \cite{peng05:_boolean,rivero00:_boolean}
 seem to be inefficient for complex topologies.
Those based on Nef polyhedra
 \cite{nef78:_beitr_zur_theor_polyed,bieri95:_nef_polyh,hachenberger07:_boolean_operat_selec_nef_compl}
 have an elegant theoretical foundation
 and is applicable to arbitrarily complex topologies,
 but they appear as an overkill for fluid modeling
 in that many elements in the modeling space of Nef polyhedra
 do not have counterparts in multiphase flows.
In addition,
 the corresponding algorithms and data structures
 are complicated and difficult to implement.
For both types of algorithms, 
 computing the topological information
 such as Betti numbers
 would be very time-consuming.

\subsection{Motivations and contributions of this work}
\label{sec:motiv-contr-this}

Methods that couple elementary concepts or tools
 from multiple disciplines
 often perform surprisingly well.
For fluid-structure interactions,
 a recent approach called isogeometric analysis
 \cite{cottrell09:_isogeom_analy,bazilevs13:_comput_fluid}
 has been increasingly popular,
 and much of its success is due to 
 the integration of finite element methods
 with highly accurate (and sometimes exact) 
 solid modeling in CAD. 
In our previous work,
 we have adopted a similar guiding principle
 to integrate IT with
 a topological space for fluid modeling
 \cite{zhang16:_mars}.
The resulting generic framework, called MARS,
 furnishes new tools for analyzing current IT methods
 \cite{zhang13:VOFadvection}
 and leads to a new IT method
 and a new curvature-estimation algorithm
 that are more accurate than current methods
 by many orders of magnitudes
 \cite{zhang17:_hfes,zhang18:_cubic_mars_method_fourt_sixth}.

In recognition of the potentially large
 benefits of integrating fluid modeling
 with multiphase flows and 
 in view of the discussions in previous subsections,
 we list a number of questions as the driving forces
 behind this work.
 \begin{enumerate}[(Q-1)]
 \item Can we propose a generic topological space
   that appropriately models physically meaningful regions
   across multiple research fields
   such as solid modeling, GIS, and multiphase flows? 
 \item Can we find a simple representation scheme
   for elements in the modeling space
   to facilitate geometric and topological queries?
 \item Can we design simple and efficient
   Boolean operations that correctly handle
   all degenerate cases?
 \item In particular,
   can we provide theoretical underpinning
   and algorithmic support for handling topological changes
   of moving regions?
 \item Meanwhile, can we extract topological information
   such as Betti numbers with optimal complexity?
 \end{enumerate}

In this paper, 
 we provide positive answers to all of the above questions.
\mbox{(Q-1)} is answered in Section \ref{sec:repr-yin-sets},
 where physically meaningful regions
 are modeled by a topological space, called the Yin space,
 which consists of regular open semianalytic sets
 with bounded boundaries.
These conditions capture fluid features
 that are commonly relevant
 in multiphase flows, solid modeling, and GIS.
Furthermore, Yin sets are
 defined in terms of computable \emph{mathematical} properties
 and are thus independent of any particular representation
 or individual application. 
As such, 
 this modeling space serves as a bridge
 between multiphase flows and the other fields
 that emphasize geometry and topology.
As our answer to (Q-2),
 each Yin set
 can be uniquely expressed 
 as the result of finite Boolean operations
 on interiors of oriented Jordan curves.
This uniqueness leads to
 an isomorphism from the Yin space
 into the Jordan space,
 a collection of certain posets of oriented Jordan curves,
 and this isomorphism reduces
 Boolean algebra on the two-dimensional Yin space
 to one-dimensional routines
 in the Jordan space,
 namely locating a point relative to a simple polygon
 and finding intersections of curve segments.
This is our answer to \mbox{(Q-3)};
 see Section \ref{sec:boolean-algebra-yin}.

In addressing (Q-4),
 we pay special attentions to
 issues related to non-manifold points on the fluid boundary,
 such as characterizing topological changes 
 with improper intersections of curves
 and dividing closed curves at these improper intersections
 to ensure correctness of Boolean operations.
However, we emphasize that,
 in both our theory and our algorithms, 
 non-manifold points of topological changes
 are treated not as an anomaly,
 but as a natural consequence of
 capturing the physical meaningfulness of fluids
 with the mathematical conditions
 that constitute the notion of Yin sets.
This is a major advantage of the Yin sets over the r-sets.

As another prominent feature of our theory,
 the number of connected components in any bounded Yin set
 is simply the number of positively oriented Jordan curves
 in its boundary representation,
 and the number of holes in a component
 is the number of negatively oriented Jordan curves
 in the boundary representation of that component.
Since these numbers are returned in $O(1)$ time,
 our answer to (Q-5) is of optimal complexity.

The rest of this paper is organized as follows.
In Section \ref{sec:prel-notat},
 we introduce prerequisites and notation.
In Section 3, we propose Yin sets
 as our fluid modeling space
 and study its topological properties.
In Section 4,
 we design Boolean operations on the Yin space
 in a way so that corresponding algorithms can be implemented
 by straightforward orchestration
 of the definitions.
Utilizing the Bentley-Ottmann paradigm of plane sweeping
 \cite{bentley79:_algor}
 in calculating intersections of curve segments, 
 our current implementation of the Boolean operations
 is close to optimal complexity.
A number of fun tests are given in Figure \ref{fig:panda-mickey} 
 to illustrate the Boolean algebra.
Finally, we draw the conclusion
 and discuss several research prospects
 in Section \ref{sec:conclusion}.



\section{Preliminaries}
\label{sec:prel-notat}


In this section
 we collect relevant definitions and theorems 
 that form the \emph{algebraic} foundation
 of this work.
Some notations introduced here
 will be repeatedly used in subsequent sections.

\subsection{Partially ordered sets}
\label{sec:posets}

The Cartesian product of a nonempty set ${\cal A}$
 with itself $n$ times is denoted
 by ${\cal A}^n$;
 in particular, ${\cal A}^0=\{\emptyset\}$.
An \emph{$n$-ary relation on} ${\cal A}$ is
 a subset of ${\cal A}^n$;
 if $n=2$ it is called a \emph{binary relation}.
A given binary relation ``$\sim$'' on a set ${\cal A}$
 is said to be an \emph{equivalence relation} 
 if and only if
 it is \emph{reflexive} ($a\sim a$),
 \emph{symmetric} ($a\sim b\ \Rightarrow\ b\sim a$),
 and \emph{transitive} ($a\sim b, b\sim c \ \Rightarrow\ a\sim c$)
 for all $a, b, c \in {\cal A}$.
A binary relation ``$\le$'' defined on a set ${\cal A}$
 is a \emph{partial order} on ${\cal A}$
 if and only if
 it is reflexive ($a\le a$),
 \emph{antisymmetric} ($a\le b, b\le a\ \Rightarrow\ b=a$),
 and transitive ($a\le b, b\le c \ \Rightarrow\ a\le c$)
 for all $a, b, c \in {\cal A}$.


A nonempty set ${\cal A}$ with a partial order $\le$ on it
  is called a \emph{partially ordered set},
  or more briefly a \emph{poset}.
Two elements $a,b\in{\cal A}$
 are \emph{comparable} if either $a\le b$ or $b\le a$;
 otherwise $a$ and $b$ are \emph{incomparable}.
If all $a,b\in{\cal A}$ are comparable by $\le$,
 then ``$\le$'' is a \emph{total order} on ${\cal A}$
 and ${\cal A}$ is a \emph{chain} or linearly-ordered set.
For examples,
 $\mathbb{R}$ with the usual order of real numbers
 is a chain;
 the \emph{power set} of ${\cal A}$,
    i.e. the set of all subsets of ${\cal A}$,
   with the subset relation ``$\subseteq$''
 is a poset but not a chain.
The notation $a\ge b$ means $b\le a$,
 and $a<b$ means both $a\le b$ and $a\ne b$.

 \begin{definition}[Covering relation]
   \label{def:covering}
   Let ${\cal A}$ denote a poset and $a,b\in{\cal A}$.
   We say $b$ \emph{covers} $a$
    and write $a\prec b$ or $b\succ a$
    if and only if $a<b$
    and no element $c\in{\cal A}$ satisfy $a<c<b$.
 \end{definition}


 Most concepts on the ordering of $\mathbb{R}$
  make sense for posets.
 Let ${\cal A}$ be a subset of a poset ${\cal P}$.
 An element $p\in {\cal P}$ is an \emph{upper bound} of ${\cal A}$
  if $a\le p$ for all $a\in {\cal A}$.
 $p\in {\cal P}$ is the \emph{least upper bound} of ${\cal A}$,
  or \emph{supremum} of ${\cal A}$ ($\sup{\cal A}$)
  if $p$ is an upper bound of ${\cal A}$,
  and $p\le b$ for any upper bound $b$ of ${\cal A}$. 
Similarly we can define the concepts of a \emph{lower bound}
 and the \emph{greatest lower bound} of ${\cal A}$
 or the \emph{infimum} of ${\cal A}$ ($\inf{\cal A}$).

 \begin{definition}[Lattice as a poset]
   \label{def:LatticeAsPoset}
   A \emph{lattice} is a poset ${\cal L}$
    satisfying that, for all $a,b\in{\cal L}$,
    both sup$\{a,b\}$ and inf$\{a,b\}$ exist in ${\cal L}$.
 \end{definition}


\subsection{Distributive Lattices}
\label{sec:distributeLattices}

An \emph{$n$-ary operation on} ${\cal A}$
 is a function $f: {\cal A}^n\rightarrow {\cal A}$
 where $n$ is the \emph{arity} of $f$.
A \emph{finitary} operation $f$ is an $n$-ary operation
 for some nonnegative integer $n\in \mathbb{N}$.
$f$ is nullary (or a constant) if its arity is zero,
 i.e. it is completely determined by
 the only element $\emptyset\in {\cal A}^0$,
 hence
 a nullary operation $f$ on ${\cal A}$
 can be identified with the element $f(\emptyset)$;
 for convenience
 it is regarded as an element of ${\cal A}$.
An operation on ${\cal A}$
 is \emph{unary} or \emph{binary}
 if its arity is 1 or 2, respectively.

 \begin{definition}[Universal algebra]
   A \emph{algebra}
    is an ordered pair $\mathbf{A}:=({\cal A}, {\cal F})$
    where ${\cal A}$ is a nonempty set
    and ${\cal F}$ a family of finitary operations on ${\cal A}$.
   The set ${\cal A}$ is
    the \emph{universe} or the \emph{underlying set} of $\mathbf{A}$
    and ${\cal F}$ the \emph{fundamental operations} of $\mathbf{A}$.
 \end{definition}

An algebra is \emph{finite}
 if the cardinality of its universe is bounded.
When ${\cal F}$ is finite, say ${\cal F}=\{f_1, f_2, \cdots, f_k\}$,
 we also write $\mathbf{A}=({\cal A}, f_1, f_2, \cdots, f_k)$
 with the operations sorted by their arities in descending order.
As a common example,
 a \emph{group} is an algebra of the form
 $\mathbf{G}=({\cal G}, \cdot, ^{-1}, 1)$
 where $\cdot, ^{-1}, 1$ 
 are a binary, a unary, and a nullary operations on ${\cal G}$,
 respectively.

 \begin{definition}[Lattice as an algebra]
   \label{def:LatticeAsAlgebraicStruct}
   A \emph{lattice} is an algebra $\mathbf{L}:=({\cal L}, {\cal F})$
    where ${\cal F}$ contains two binary operations $\vee$ and $\wedge$
    (read ``join'' and ``meet'' respectively)
    on ${\cal L}$ that satisfy the following axiomatic identities
    for all $x,y,z\in {\cal L}$,
    \begin{enumerate}[({LA}-1)]
    \item commutative laws: $x\vee y = y\vee x$,
      $x\wedge y = y\wedge x$;
    \item associative laws: $x\vee (y\vee z) = (x\vee y)\vee z$,
      $x\wedge (y\wedge z) = (x\wedge y)\wedge z$;
    \item absorption laws: $x= x \vee (x \wedge y)$,
      $x= x \wedge (x \vee y)$.
    \end{enumerate}

 \end{definition}

Sometimes the following idempotent laws
  are also included in the definition of a lattice
  although they can be derived
  from the above three axioms,
\begin{equation}
  \label{eq:idempotentLaws}
  x\vee x =x, \qquad x\wedge x =x.
\end{equation}

A lattice defined as a poset
 can be converted to an algebra
 by constructing the binary operations
 as $a\vee b=\sup\{a,b\}$ and $a\wedge b=\inf\{a,b\}$;
 the converse case can also be achieved
 by defining the partial order
 as $a\le b$ $\Leftrightarrow$ $a=a\wedge b$.
Hence
 \textsc{Definition} \ref{def:LatticeAsPoset}
 and \textsc{Definition} \ref{def:LatticeAsAlgebraicStruct}
 are equivalent.

 \begin{definition}
   \label{def:boundedLattice}
   A \emph{bounded lattice}
    is an algebra $({\cal L}, {\cal F})$
    where ${\cal F}$ contains
    binary operations $\vee, \wedge$
    and nullary operations $\Zero, \One$
    so that $({\cal L}, \vee, \wedge)$ is a lattice
    and, $\forall x \in {\cal L}$,
    \begin{equation}
      \label{eq:boundedLattice}
      x\wedge \Zero=\Zero,\qquad x\vee \One=\One.
    \end{equation}
 \end{definition}

The boundedness in the above definition
 is best understood from the poset viewpoint:
 $\Zero\le x$
 and $x\le \One$ for all $x\in {\cal L}$.

 \begin{definition}
   \label{def:distributiveLattice}
   A \emph{distributive lattice} is a lattice
    which satisfies either
    of the distributive laws
    \begin{equation}
      \label{eq:distributiveLaws}
      x\wedge (y\vee z) = (x\wedge y)\vee (x\wedge z),
      \qquad
      x\vee (y\wedge z) = (x\vee y)\wedge (x\vee z).
    \end{equation}
 \end{definition}
Either identity in (\ref{eq:distributiveLaws})
 can be deduced from the other 
 and \textsc{Definition} \ref{def:LatticeAsAlgebraicStruct}
 \cite[p. 10]{burris81:_cours_univer_algeb}.

A notion central to every branch of mathematics
 is isomorphism.
In particular,
 two lattices are isomorphic if they have
 the same structure.

 \begin{definition}[Lattice isomorphism]
   A \emph{homomorphism} of the lattice
    $({\cal L}_1, \vee, \wedge)$ into 
    the lattice $({\cal L}_2, \cup, \cap)$ 
    is a map $\phi:{\cal L}_1\rightarrow {\cal L}_2$
    satisfying
    \begin{equation}
      \label{eq:homomorphismOfLattices}
      \forall x,y \in {\cal L}_1,\ \
      \phi(x\vee y) = \phi(x) \cup \phi(y),
      \ \phi(x\wedge y) = \phi(x) \cap \phi(y).
    \end{equation}
    An \emph{isomorphism} is a bijective homomorphism.
 \end{definition}

More details on distributive lattices
 can be found in \cite{burris81:_cours_univer_algeb}
 from the perspective of universal algebra
 and in \cite[ch. 3]{stanley12:_enumer_combin} 
 from the viewpoint of posets.
See \cite{graetzer09:_lattic_theor} for a more accessible one.

\subsection{Boolean algebra}
\label{sec:BooleanAlgebra}

The simplest definition 
 may be due to Huntington 
 \cite{huntington04:_sets_indep_postul_algeb_logic}.

 \begin{definition}
   \label{def:BooleanAlgebra}
   A \emph{Boolean algebra} is an algebra of the form
   \begin{equation}
     \label{eq:BooleanAlgebra}
     \mathbf{B}:=({\cal B},\ \vee,\ \wedge,\ \, ',\ \Zero,\ \One),
   \end{equation}
    where the binary operations $\vee,\ \wedge$,
    the unary operation $\, '$ called complementation,
    and the nullary operations $\Zero,\ \One$
    satisfy
    \begin{enumerate}[({BA}-1)]
    \item the identity laws: $x\wedge \One = x$,\  $x\vee \Zero=x$,
    \item the complement laws: $x\wedge x' = \Zero$,\  $x\vee x'=\One$,
    \item the commutative laws (LA-1),
    \item the distributive laws (\ref{eq:distributiveLaws}).
    \end{enumerate}
 \end{definition}

Other definitions contain redundant axiomatic laws 
 that can be deduced from the above four conditions.
For example,
 Givant and Halmos \cite[p. 10]{givant09:_introd_boolean_algeb}
 defined a Boolean algebra as an algebra
 with its fundamental operations
 satisfying (LA-1), (LA-2),
 (\ref{eq:idempotentLaws}), (\ref{eq:boundedLattice}),
 (\ref{eq:distributiveLaws}),
 (BA-1), (BA-2), $\Zero'=\One$, $\One'=\Zero$, $(x')'=x$,
 and the DeMorgan's laws,
 \begin{equation}
   \label{eq:DeMorgansLaws}
   (x\wedge y)' = x'\vee y',\qquad
   (x\vee y)' = x'\wedge y'.
 \end{equation}

In this work we adopt the viewpoint
 of Burris and Sankappanavar \cite[p. 116]{burris81:_cours_univer_algeb}.
 \begin{definition}
   A \emph{Boolean algebra} is a bounded distributive lattice
   with an additional complementation operation
   that satisfies the complement laws (BA-2).
 \end{definition}

\subsection{Veblen's theorem}

A \emph{graph} is an ordered pair
 $G=(V, E)$ where $V$ is the set of \emph{vertices}
 and $E$ the set of \emph{edges},
 each edge being an \emph{unordered} pair of
 distinct vertices.
$G'$ is a \emph{subgraph} of $G$,
 written as $G'\subseteq G$,
 if $V(G')\subseteq V(G)$ and $E(G')\subseteq E(G)$.
If $uv\in E(G)$,
 then $u$ and $v$ are \emph{adjacent} in $G$,
 and the edge $uv$ is said to be 
 \emph{incident} to $u$ and $v$.
The \emph{degree} of a vertex $v$ is
 the number of edges incident to $v$.


\begin{definition}
  \label{def:cycles}
  A \emph{path} is 
   a graph $P$ of the form
  \mbox{$V(P)=\{v_0, v_1, \ldots, v_{\ell}\}$} and
  $E(P)=\{v_0v_1, v_1v_2, \ldots, v_{\ell-1}v_{\ell}\}$.
  A \emph{cycle} is a graph of the form $C:= P + v_0v_{\ell}$
   where $P$ is a path and $\ell\ge 2$.
\end{definition}

A graph $G$ is \emph{connected}
 if for every pair of distinct vertices in $V(G)$
 there is a subgraph of $G$
 as the path from one vertex to the other.
A \emph{component} of the graph is a \emph{maximal connected subgraph}.


\begin{theorem}[Veblen
  \cite{veblen13:_applic_modul_equat_analy_situs}]
  \label{thm:Veblen}
  The edge set of a graph can be partitioned into
  edge-disjoint cycles
  if and only if the degree of every vertex is even.
\end{theorem}
\begin{proof}
  If a graph is the union of a number of edge-disjoint cycles,
   then clearly a vertex contained in $n$ cycles has degree $2n$.
  Hence the necessity holds.
  
  Suppose that the degree of every vertex is a positive even integer.
  How do we find a single cycle in $G$?
  Let $P=x_0 x_1 \cdots x_{\ell}$ be a path of maximal length $\ell$
   in $G$.
  Since $d(x_0)\ge 2$, $x_0$ must have another neighbor $y$
   in addition to $x_1$.
  Furthermore, we must have $y=x_i$ for some $i\in[2,\ell]$;
   otherwise it would contradict the starting condition
   that $P$ is of maximal length.
  Therefore we have found a cycle $x_0 x_1 \cdots x_i$.

  Having found one cycle, we remove it from $G$.
  If the remaining subgraph $G_1$ of $G$ is not empty,
   then the degree of every vertex in $G_1$
   remains positive and even.
  Repeating the cycle-finding procedures completes the proof;
   see \cite[p. 5]{bollobas08:_moder_graph_theor}.
\end{proof}

A \emph{multigraph} is an augmented graph
 that allows \emph{loop}s
 and \emph{multiple edges};
 the former is defined as a special edge joining a vertex to itself
 and the latter several edges joining the same vertices.
A loop contributes 2 to the degree of a vertex
 while each edge in a multiple edge contribute 1.
As for cycles,
 the condition of $\ell\ge 2$ 
 in \textsc{Definition} \ref{def:cycles}
 is changed to $\ell\ge 0$ for a multigraph:
 $\ell=0$ indicates a loop
 and $\ell=1$ two edges joining the same vertices.
It is straightforward to extend Theorem \ref{thm:Veblen}
 to multigraphs.

 \begin{theorem}
   \label{thm:VeblenMultigraph}
   The edge set of a multigraph can be partitioned
    into edge-disjoint cycles
    if and only if the degree of each vertex is positive and even.
 \end{theorem}

A \emph{directed} graph/multigraph
 is a graph/multigraph where the edges
 are \emph{ordered} pairs of vertices.
An edge $uv$ is then said to \emph{start} at $u$
 and \emph{end} at $v$.
Furthermore,
 the degree of a vertex $v$ 
 is split into
 the \emph{outdegree} $d^+(v)$ and the \emph{indegree} $d^-(v)$,
 with the former as the number of edges starting at $v$
 and the latter that of edges ending at $v$.
Veblen's Theorem generalizes to directed multigraphs
 in a straightforward manner.

 \begin{theorem}
   \label{thm:VeblenMultigraphDirected}
   The edge set of a directed multigraph can be partitioned
    into directed cycles
    if and only if each vertex has the same outdegree
    and indegree.
 \end{theorem}



\section{Yin sets}
\label{sec:repr-yin-sets}

Based on regular semianalytic sets
 introduced in Section \ref{sec:regularSets},
 we propose in Section \ref{sec:YinSpace}
 the Yin space for fluid modeling in two dimensions.
From the viewpoint of Jordan curves in Section \ref{sec:JordanCurves},
 we study in Section \ref{sec:topologyOfYinSets}
 the local and global topology of Yin sets,
 the results of which yield the notion of \emph{realizable spadjors}
 in Section \ref{sec:spadjors}
 as a unique boundary representation of Yin sets.

\subsection{Regular semianalytic sets}
\label{sec:regularSets}

In a topological space ${\mathcal X}$,
 the \emph{complement} of a subset ${\mathcal P}\subseteq {\mathcal X}$,
 written ${\mathcal P}'$,
 is the set ${\mathcal X}\setminus {\mathcal P}$.
The \emph{closure} of a set ${\mathcal P}\subseteq{\mathcal X}$,
 written ${\mathcal P}^-$,
 is the intersection of all closed 
 supersets of ${\mathcal P}$.
The \emph{interior} of ${\mathcal P}$, written ${\mathcal P}^{\circ}$,
 is the union of all open subsets of ${\mathcal P}$.
The \emph{exterior} of ${\mathcal P}$,
 written ${\mathcal P}^{\perp}:= {\mathcal P}^{\prime\circ}
 :=({\mathcal P}')^{\circ}$,
 is the interior of its complement.
By the identity ${\mathcal P}^- ={\mathcal P}^{\prime\circ\prime}$
 \cite[p. 58]{givant09:_introd_boolean_algeb},
 we have 
 ${\mathcal P}^{\perp}={\mathcal P}^{-\prime}$.
A point $\mathbf{x}\in {\mathcal X}$ is
 a \emph{boundary point} of ${\mathcal P}$
 if $\mathbf{x}\not\in {\mathcal P}^{\circ}$
 and $\mathbf{x}\not\in {\mathcal P}^{\perp}$.
The \emph{boundary} of ${\mathcal P}$, written $\partial {\mathcal P}$,
 is the set of all boundary points of ${\mathcal P}$.
It can be shown that
 ${\mathcal P}^{\circ}={\mathcal P}\setminus \partial {\mathcal P}$
 and
 ${\mathcal P}^-= {\mathcal P}\cup \partial {\mathcal P}$.
An open set ${\mathcal P}\subseteq {\mathcal X}$ is \emph{regular} if
 it coincides with the interior of its own closure,
 i.e. if ${\mathcal P}={\mathcal P}^{-\circ}$.
A closed set ${\mathcal P}\subseteq {\mathcal X}$ is \emph{regular} if
 it coincides with the closure of its own interior,
 i.e. if ${\mathcal P}={\mathcal P}^{\circ-}$.
The duality of the interior and closure operators
 implies
 ${\mathcal P}^{\circ} ={\mathcal P}^{\prime-\prime}$,
 hence ${\mathcal P}$ is a \emph{regular open set}
 if and only if ${\mathcal P}={\mathcal P}^{\perp\perp}
 :=\left({\mathcal P}^{\perp}\right)^{\perp}$.
For any subset $Q\subseteq{\mathcal X}$,
 it can be shown that $Q^{\perp\perp}$ is a regular open set
 and $Q^{\circ -}$ is a regular closed set.

Regular sets, open or closed, capture the salient feature 
 that physically meaningful regions
 are free of lower-dimensional elements
 such as isolated points and curves in 2D
 and dangling faces in 3D.

 \begin{theorem}[MacNeille \cite{macneille37:_partial}
   and Tarski \cite{tarski37:_ueber_mengen_mengen}]
   \label{thm:regularOpenAlgebra}
   Let ${\mathbb B}$ denote
    the class of all regular open sets
    of a topological space ${\mathcal X}$
    and define
    ${\mathcal P}\cup^{\perp\perp}{\mathcal Q}
    := ({\mathcal P}\cup {\mathcal Q})^{\perp\perp}$
    for all ${\mathcal P},{\mathcal Q}\subseteq {\mathcal X}$.
   Then
    $\mathbf{B}_o:=({\mathbb B},\ \cup^{\perp\perp},\ \cap,\ \, ^{\perp},\
    \emptyset,\ {\mathcal X})$
    is a Boolean algebra.
 \end{theorem}
 \begin{proof}
   See \cite[\S 10]{givant09:_introd_boolean_algeb}.
 \end{proof}

Similarly, it can be shown that,
 with appropriately defined operations,
 regular closed sets
 of a topological space ${\mathcal X}$
 also form a Boolean algebra
 \cite[p. 39]{kuratowski76:_set_theor_introd_descr_set_theor}.

Regular sets are not perfect for representing
 physically meaningful regions yet: 
 some of them
 cannot be described by a finite number of symbol structures.
For example, some sets
 have nowhere differentiable boundaries,
 which, in their parametric forms,
 are usually infinite series of continuous functions
 \cite{sagan92:_elemen_proof_schoen_space_fillin}.
Another pathological case
 is more subtle:
 intersecting two regular sets with piecewise smooth boundaries
 may yield an infinite number of disjoint regular sets.
Consider 
 \begin{equation}
   \label{eq:pathologicalIntersection}
   \left\{
   \begin{array}{l}
   {\mathcal A}_p := \{(x,y)\in \mathbb{R}^2 :
    -2< y< \sin\frac{1}{x},\ 
    0< x< 1 \},\\
   {\mathcal A}_s := \{(x,y)\in \mathbb{R}^2 :
    0 < y < 1,\ 
    -1< x< 1 \}.
   \end{array}
   \right.
 \end{equation}

Although both ${\mathcal A}_p$ and ${\mathcal A}_s$
 are described by two inequalities,
 their intersection
 is a disjoint union of an infinite number of regular sets;
 see
 \cite[Fig. 4-1, Fig. 4-2]{requicha77:_mathem_model_rigid_solid_objec}.
This poses a fundamental problem
 that results of Boolean operations of two regular sets
 may not be well represented on a computer
 by a finite number of entitites.

Therefore, 
 we need to find a proper subspace of regular sets,
 each element of which is finitely describable.
This search eventually
 arrives at semianalytic sets.

\begin{definition}
  \label{def:semianalytic-sets}
  A set ${\mathcal S}\subseteq\mathbb{R}^{\Dim}$ is \emph{semianalytic}
   if there exist a finite number of 
   analytic functions $g_i:\mathbb{R}^{\Dim}\rightarrow \mathbb{R}$
   such that ${\mathcal S}$ is in the universe
   of a finite Boolean algebra formed from the sets
   \begin{equation}
     \label{eq:semiAnalyticForm}
     {\mathcal X}_i=\left\{\mathbf{x}\in \mathbb{R}^{\Dim}
       :  g_i(\mathbf{x})\ge 0\right\}.
   \end{equation}
  The $g_i$'s are called the \emph{generating functions}
   of ${\mathcal S}$.
  In particular, a semianalytic set is \emph{semialgebraic}
  if all of its generating functions are polynomials.
\end{definition}

Recall that a function is \emph{analytic}
 if and only if
 its Taylor series at $\mathbf{x}_0$
 converges to the function in some neighborhood
 for every $\mathbf{x}_0$ in its domain.
In the example of (\ref{eq:pathologicalIntersection}),
 ${\mathcal A}_s$ is semianalytic while 
 ${\mathcal A}_p$ is not,
 because the Taylor series of $g_2(x,y):=\sin\frac{1}{x}-y$
 at the origin does not converge.
Roughly speaking,
 the boundary curves of regular semianalytic sets
 are piecewise smooth.

\subsection{$\mathbb{Y}$: 
  the Yin space for fluid modeling}
\label{sec:YinSpace}

Regular closed semianalytic sets have been
 an essential mathematical tool
 for solid modeling
 since the dawning time of this field
 \cite{requicha78:_mathem_found_const_solid_geomet}.
However,
 as shown in Figure \ref{fig:boundaryDecompositions}, 
 requiring regular semianalytic sets to be closed
 would make their boundary representation
 not unique,
 degrading the isomorphism in Definition
 \ref{def:boundaryToInteriorMap}
 and Theorem \ref{thm:b2iMapIsBijective}
 to a homomorphism.
As another work closely related to this one,
 the analysis on a family of interface tracking methods
 \cite{zhang13:VOFadvection,zhang16:_mars}
 via the theory of donating regions
 \cite{zhang13:_donat_region,zhang15:_gener_donat_region}
 also requires that the regular sets be open.
Therefore,
 only regular open semianalytic sets are employed in this work.

\begin{definition}
  \label{def:YinSet}
  A \emph{Yin set}\footnote{
    Yin sets are named after the first author's mentor,
    Madam Ping Yin.
    As a coincidence,
     the most important dichotomy in Taoism
     consists of Yin and Yang,
     where Yang represents the active, the straight, the ascending,
     and so on, 
     while Yin represents the passive, the circular, the descending,
     and so on.
  From this viewpoint, straight lines and Jordan curves can be 
   considered as Yang 1-manifolds and Yin 1-manifolds, respectively.
  } $\mathcal{Y}\subseteq \mathbb{R}^2$
   is a regular open semianalytic set
   whose boundary is bounded.
  The class of all such Yin sets form
   the \emph{Yin space} $\mathbb{Y}$.
\end{definition}

 \begin{figure}
   \centering
   \subfigure[an unbounded Yin set]{
     \includegraphics[width=0.4\linewidth]{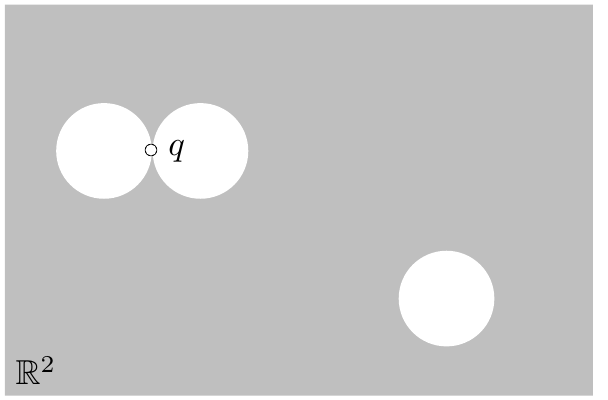}
   }
   \hfill
   \subfigure[a bounded Yin set]{
     \includegraphics[width=0.4\linewidth]{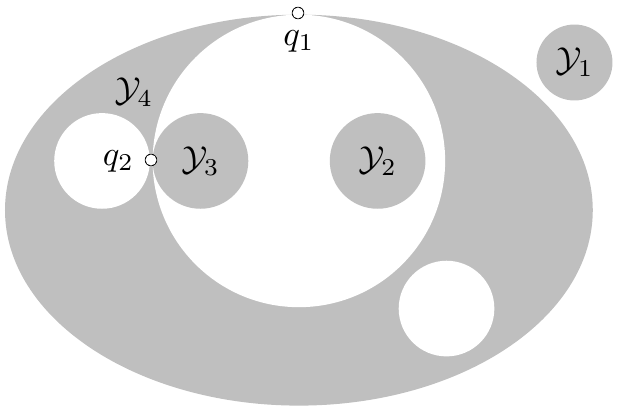}
   }
   \caption{Examples of Yin sets.
     The Yin set in (a) is obtained
      by removing from $\mathbb{R}^2$ three closed balls,
      two of which share a common boundary point $q$.
     The Yin set in (b) is 
      the union of four pairwise disjoint Yin sets
      ${\cal Y}=\bigcup_{i=1}^4{\cal Y}_i$,
      where ${\cal Y}_4$ is an open ellipse
      with three closed balls removed,
      two of which share a common boundary point $q_2$.
     The points $q$, $q_1$, $q_2$
      are boundary points but not interior points
      of the Yet sets.
   }
 \label{fig:YinSets}
\end{figure}

In Figure \ref{fig:YinSets},
 the Yin set in subplot (a)
 is unbounded and connected
 while that in subplot (b)
 is bounded and consists of four disjoint components.

By Definition \ref{def:semianalytic-sets},
 semianalytic sets are closed under
 set complementation, finite union, and finite intersection.
Then by Theorem \ref{thm:regularOpenAlgebra},
 regular open semianalytic sets
 form a Boolean algebra
 since they are the intersection of the universes
 of two Boolean algebras.
Furthermore,
 Boolean operations on Yin sets
 preserve the attribute of
 a bounded boundary being bounded,
 hence we have
 
 \begin{theorem}
   \label{thm:YinSetsFormABooleanAlgebra}
    The algebra
    $\mathbf{Y}:=({\mathbb Y},\ \cup^{\perp\perp},\ \cap,\ \, ^{\perp},\
    \emptyset,\ \mathbb{R}^2)$
    is a Boolean algebra.
 \end{theorem}


\subsection{Jordan curves and orientations}
\label{sec:JordanCurves}

A \emph{path} in $\mathbb{R}^2$ from $p$ to $q$
 is a continuous map $f:[0,1]\rightarrow \mathbb{R}^2$
 satisfying $f(0)=p$ and $f(1)=q$.
A subset ${\cal P}$ of $\mathbb{R}^2$
 is \emph{path-connected} if
 every pair of points of ${\cal P}$
 can be joined by a path in ${\cal P}$.
Given ${\cal Q}\subset \mathbb{R}^2$,
 define an equivalence relation on ${\cal Q}$
 by setting $x\sim y$ if there is a path-connected subset of $Q$
 that contains both $x$ and $y$.
The equivalence classes are called
 the \emph{path-connected component}s of ${\cal Q}$.

A \emph{planar curve} 
 is a continuous map
 $\gamma: (0,1) \rightarrow \mathbb{R}^{2}$.
It is \emph{smooth} if the map is smooth.
It is \emph{simple} if the map is injective;
 otherwise it is \emph{self-intersecting}.
Although strictly speaking a curve $\gamma$ is a map,
 we also use $\gamma$ to refer to its image.
Two distinct piecewise smooth curves $\gamma_1$ and $\gamma_2$
 \emph{intersect} at $q$ if there exist $s_1,s_2\in (0,1)$
 such that \mbox{$\gamma_1(s_1)=\gamma_2(s_2)=q$}.
Then $q$ is the \emph{intersection} of $\gamma_1$ and $\gamma_2$.
For an open ball ${\cal N}_r(q)$ with sufficiently small radius $r$,
 ${\cal N}_r(q)\setminus\gamma_1$ consists of
 two disjoint connected regular open sets.
If $\gamma_2\setminus q$ is 
 entirely contained in one of these two sets,
 $q$ is an \emph{improper intersection};
 otherwise it is a \emph{proper intersection}.
%
Two curves are \emph{disjoint} if
 they have neither proper intersections nor improper ones.
Suppose upon its extension to a path,
 a simple curve $\gamma$ further satisfies
 $\gamma(0)=\gamma(1)$,
 then $\gamma$ is
 a \emph{simple closed curve} or \emph{Jordan curve}. 

\begin{theorem} [Jordan Curve Theorem \cite{jordan87:_cours_daanl_polyt}]
 \label{thm:jordan}
 The complement of a Jordan curve $\gamma$
 in the plane $\mathbb{R}^2$
 consists of two components,
 each of which has $\gamma$ as its boundary.
 One component is bounded and the other is unbounded;
 both of them are open and path-connected.
\end{theorem}

The above theorem states that a Jordan curve divides the plane
 into three parts: itself, its \emph{interior},
 and \emph{exterior}.

\begin{definition}
\label{def:interior}
The \emph{interior} of an oriented Jordan curve $\gamma$,
 denoted by $\Int(\gamma)$,
 is the component of the complement of $\gamma$
 that always lies to the left
 when an observer traverses the curve
 in the increasing direction of the parameterization $s$.
\end{definition}

A Jordan curve is said to be \emph{positively oriented}
 if its interior is the bounded component of its complement;
 otherwise it is \emph{negatively oriented}.
The orientation of a Jordan curve
 can be flipped
 by reversing the increasing direction of the parameterization.
The following notion will be used throughout this work.

\begin{definition}
  \label{def:almostDisjoint}
  Two Jordan curves are 
   \emph{almost disjoint}
   if they have no proper intersections
   and at most a finite number of improper intersections.
\end{definition}

\subsection{The local and global topology of a Yin set}
\label{sec:topologyOfYinSets}

The following lemma characterizes
 the local topology of a Yin set at its boundary.

 \begin{figure}
   \centering
   \subfigure[one curve]{
     \includegraphics[width=0.20\linewidth]{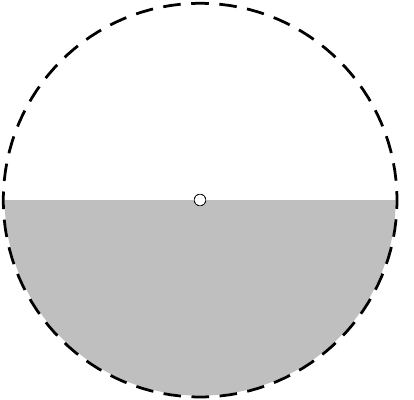}
   }
   \hfill
   \subfigure[two curves]{
     \includegraphics[width=0.20\linewidth]{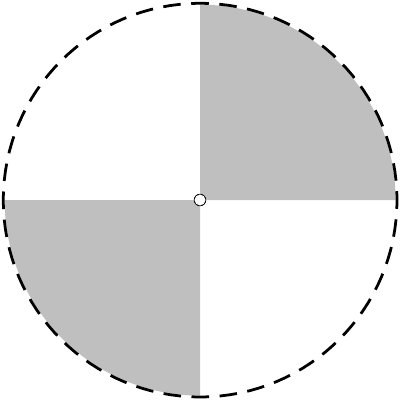}
   }
   \hfill
   \subfigure[three or more]{
     \includegraphics[width=0.20\linewidth]{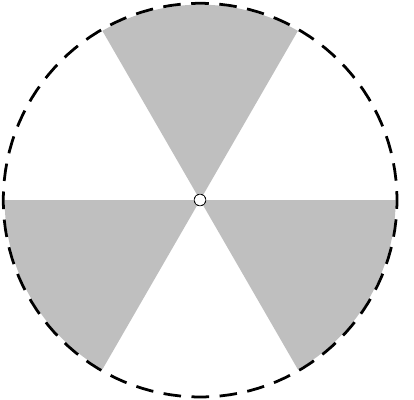}
   }
   \caption{The local topology at a boundary point $p$
     (open dot) of a Yin set ${\cal Y}$ (shaded region).
     The dashed circle represents
      the boundary of ${\cal N}_r(p)$,
      a local neighborhood of $p$. 
     The open dot in subplot (c) corresponds to the boundary point $q_2$
      in Figure \ref{fig:YinSets}(b)
      while that in subplot (b) to $q$, $q_1$
       in Figure \ref{fig:YinSets}(a), (b).
     All other boundary points in Figure \ref{fig:YinSets}
      correspond to that in subplot (a).
   }
 \label{fig:localTopology}
\end{figure}

 \begin{lemma}
   \label{lem:localTopologyOfBdPtOfYinSet}
   Let $p\in\partial {\cal Y}$ be a boundary point of a Yin set
    ${\cal Y}\subset \mathbb{R}^2$
    and denote by ${\cal N}_r(p)$ the open ball centered at $p$
    with its radius $r>0$.
   For any sufficiently small $r$,
    \begin{enumerate}[(a)]
    \item $\partial {\cal Y}\cap {\cal N}_r(p)$
      consists of $n_c(p)$ simple curves, 
      where $n_c(p)$ is a finite positive integer,
    \item if $n_c(p)>1$,
      all simple curves in (a) intersect
      and $p$ is their sole intersection,
    \item ${\cal N}_r(p)\setminus\partial {\cal Y}$
      consists of an even number of disjoint regular open sets;
      for two such sets sharing a common boundary,
      one is a subset of ${\cal Y}$
      while the other that of ${\cal Y}^{\perp}$.
    \end{enumerate}
 \end{lemma}
 \begin{proof}
   Since ${\cal Y}$ is semianalytic,
    Definition \ref{def:semianalytic-sets} implies that
    ${\cal Y}\cap {\cal N}_r(p)$ is defined
     by a finite number of analytic functions
     $g_i:\mathbb{R}^2\rightarrow \mathbb{R}$.
   By the implicit function theorem,
    each $g_i(\mathbf{x})=0$ defines a planar curve.
   Then (a) and (b)
    follows from the condition of ${\cal Y}$ being regular open
    and the condition that $r>0$ can be as small as one wishes.
   Hence the local topology at a boundary point $p$
    can be characterized by the number of the aforementioned curves
    that intersect at $p$,
    as is shown in Figure \ref{fig:localTopology}.

   By (a), (b),
    and the fact of ${\cal N}_r(p)$ being regular open,
    ${\cal N}_r(p)\setminus\partial {\cal Y}$
    consists of an even number of disjoint regular open sets.
   Consider two such sets that share a common boundary.
   Suppose both of them are subsets of ${\cal Y}$,
    then it contradicts the fact of ${\cal Y}$ being regular.
   Suppose both of them are subsets of ${\cal Y}^{\perp}$,
    then it contradicts the fact of their common boundary
    being a subset of $\partial{\cal Y}$.
   Hence (c) follows.
 \end{proof}


 \begin{figure}
   \centering
   \includegraphics[width=0.2\linewidth]{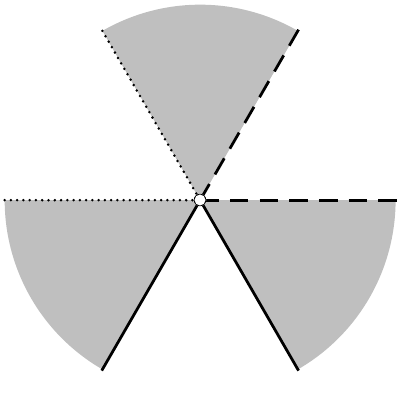}
   \caption{Decomposing the boundary $\partial {\cal Y}$
     of a connected Yin set ${\cal Y}$
     into a set of Jordan curves by disentangling
     the multiple boundary curves
     that intersect at a boundary point $q$ (the hollow dot).
     This example corresponds to
      Figure \ref{fig:localTopology} (b) and (c).
     The shaded fan-shaped wedges represent
      connected components of
      ${\cal Y}\cap {\cal N}_r(q)$,
      and the white fan-shaped wedges those of
      ${\cal Y}^{\perp}\cap {\cal N}_r(q)$.
     Two simple curves incident to $q$
      are assigned to the same Jordan curve
      if they are part of the boundary
      of the same component of ${\cal Y}^{\perp}\cap {\cal N}_r(p)$.
   }
 \label{fig:edgePairing}
\end{figure}

 \begin{figure}
   \centering
   \subfigure[The boundary Jordan curves have proper intersections]{
     \includegraphics[width=0.3\linewidth]{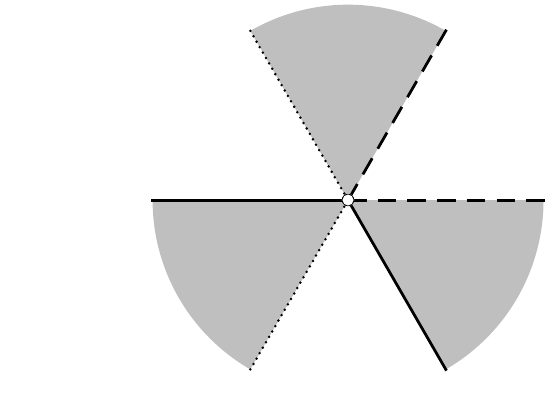}
   }
   \hfill
   \subfigure[${\cal Y}$ becomes disconnected]{
     \includegraphics[width=0.4\linewidth]{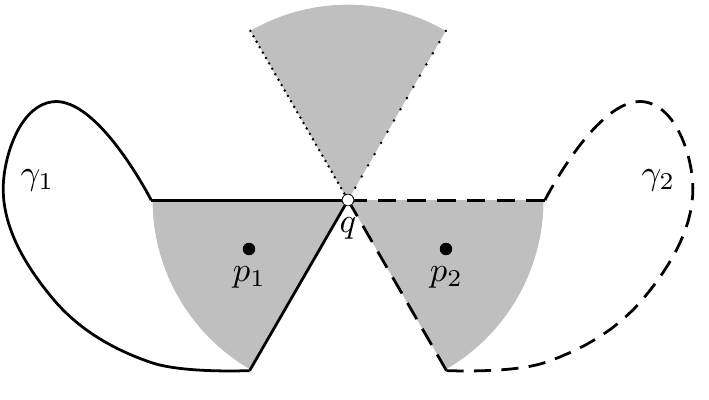}
     }
   \caption{The decomposition method shown in Figure \ref{fig:edgePairing}
     is the only valid choice to decompose
     the boundary of a \emph{connected} Yin set into 
     pairwise almost disjoint Jordan curves,
     because any other choice yields a contradiction.
   }
 \label{fig:edgePairingProof}
\end{figure}

The above lemma on the local topology
 naturally yields a result on the global topology
 of a connected Yin set.

 \begin{theorem}
   \label{thm:uniqueRep}
   For a connected Yin set ${\cal Y}\ne \emptyset, \mathbb{R}^2$,
    its boundary $\partial {\cal Y}$ can be uniquely partitioned
    into a finite set of pairwise almost disjoint Jordan curves.
 \end{theorem}
 \begin{proof}
   Without loss of generality,
    we focus on a single connected component of $\partial {\cal Y}$.
   By Lemma \ref{lem:localTopologyOfBdPtOfYinSet},
    a boundary point $q\in\partial{\cal Y}$
    can be classified into two types
    according to $n_c$,
    the number of curves that intersect at $q$.

   If all boundary points satisfy $n_c=1$,
    the condition of $\partial {\cal Y}$ being bounded
    implies that
    this connected component of $\partial {\cal Y}$
    must be a single Jordan curve.
   Hence the statement holds trivially.

   Otherwise a finite number of boundary points satisfy $n_c>1$.
   Then we construct a multigraph $G_{\partial {\cal Y}}$
    by setting its vertex set as $\{q\}$ 
    and by obtaining its edges from 
    dividing $\partial {\cal Y}$ with $\{q\}$.
   By Lemma \ref{lem:localTopologyOfBdPtOfYinSet} (a), (b),
    and $\partial {\cal Y}$ being connected,
    the degree of each vertex in ${G}_{\partial {\cal Y}}$
    is even. 
   Then it follows from Theorem \ref{thm:VeblenMultigraph} that
    we can decompose $\partial {\cal Y}$ into
    edge-disjoint cycles.
   As shown in Figure \ref{fig:edgePairing},
    the decomposition is performed by requiring that, at each vertex $q$,
    {two edges incident to $q$ are assigned to the same cycle
    if and only if they belong to the boundary
    of the same connected component of ${\cal Y}^{\perp}\cap {\cal N}_r(q)$}.
   Consequently,
    no edges in different cycles intersect properly
    at each self-intersection, 
    hence the resulting Jordan curves are pairwise almost disjoint.

   Finally,
    we show that the decomposition in Figure \ref{fig:edgePairing}
    is the only valide choice. 
   Consider the other possibilities
    shown in Figure \ref{fig:edgePairingProof}.
   If the two edges in the same cycle are not adjacent
    as in Figure \ref{fig:edgePairingProof} (a),
    then two cycles would have a proper intersection at $q$,
    which contradicts the condition of the Jordan curves
    having no proper intersections.
   As for the last possibility
    shown in Figure \ref{fig:edgePairingProof} (b),
    two edges in the same cycle are adjacent
    but they both belong to the boundary of some connected component
    of ${\cal Y}\cap {\cal N}_r(p)$.
   Then we can draw Jordan curves 
    ${\gamma}_1, {\gamma}_2\subset \partial {\cal Y}$
    that contain them.
   The simpleness of a Jordan curve implies
    ${\gamma}_1 \ne {\gamma}_2$.
   Then there exist two points $p_1,p_2\in {\cal Y}$
    that belong to the bounded complements
    of ${\gamma}_1$ and ${\gamma}_2$, respectively.
   Because ${\cal Y}$ is connected,
    there exists a path within ${\cal Y}$
    that connects $p_1$ and $p_2$.
   By Theorem \ref{thm:jordan},
    this path has to intersect ${\gamma}_1$
    at some point, say $p_c$.
   The construction of this path implies
    $p_c\in {\cal Y}$;
    the construction of ${\gamma}_1$
    implies $p_c\in \partial {\cal Y}$.
   This is a contradiction because ${\cal Y}$ is open.
  \end{proof}

The above proof hinges on the fact
 of a Yin set being open,
 so Theorem \ref{thm:uniqueRep} may not hold
 for the closure of a Yin set.
As shown in Figure \ref{fig:boundaryDecompositions},
 the decomposition of the boundary of a regular closed set 
 is \emph{not} unique.
This is a main reason that
  we do not model physically meaningful regions
  with regular closed sets.

 \begin{figure}
   \centering
   \subfigure[${\cal Y}={\cal Y}_1\cup{\cal Y}_2$;
   ${\cal Y}_1\cap{\cal Y}_2=\emptyset$; $q_1,q_2\not\in {\cal Y}$]{
     \includegraphics[width=0.38\linewidth]{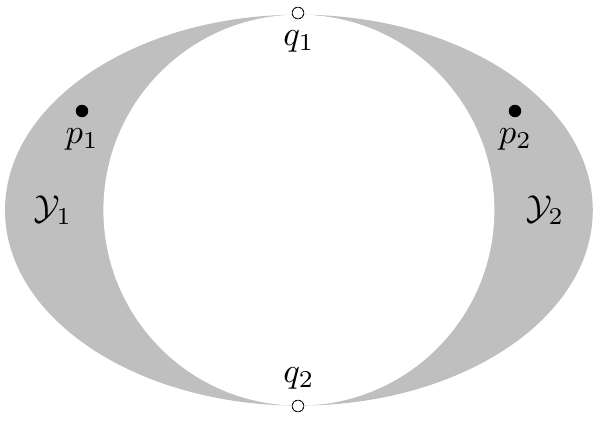}
   }
   \hfill
   \subfigure[${\cal Y}^-={\cal Y}_1^-\cup{\cal Y}_2^-$;
    $q_1,q_2\in {\cal Y}^-$
   ]{
     \includegraphics[width=0.38\linewidth]{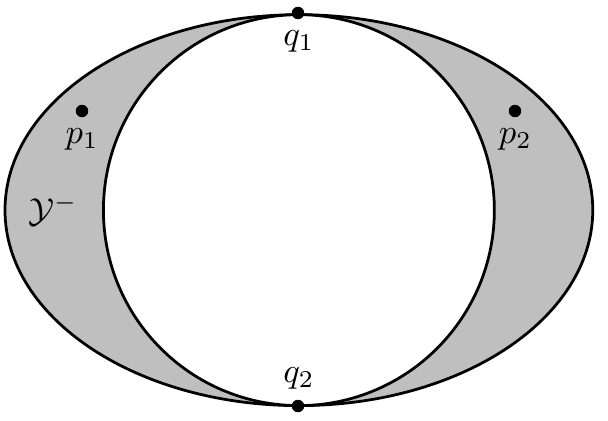}
   }
   \caption{Theorem \ref{thm:uniqueRep} does not hold
     for the closure of a Yin set.
     In (a), the Yin set ${\cal Y}$ consists of two disjoint components
     ${\cal Y}_1$ and ${\cal Y}_2$
     that share two common boundary points $q_1$ and $q_2$.
     In (b), the closure of ${\cal Y}$,
      ${\cal Y}^-={\cal Y}\cup\partial {\cal Y}$,
      becomes connected.
     The solid curves indicate that ${\cal Y}^-$ is a regular closed set.
     The crucial difference is that,
      after the closure of ${\cal Y}$, 
      the two points $p_1,p_2$ that previously
      belong to the two disjoint Yin sets in (a)
      can now be joined by a path in ${\cal Y}^-$.
     Consequently,
      the decomposition of $\partial{\cal Y}^-$ 
       into a set of pairwise almost disjoint Jordan curves
       is not unique any more.
   }
 \label{fig:boundaryDecompositions}
\end{figure}

To relate ${\cal Y}$ to its boundary Jordan curves,
 we first define a partial order on Jordan curves.

 \begin{definition}[Inclusion of Jordan curves]
   \label{def:inclusionRelation}
   A Jordan curve $\gamma_k$
    is said to \emph{include} $\gamma_{\ell}$,
    written as
    $\gamma_{k}\ge \gamma_{\ell}$ or $\gamma_{\ell}\le \gamma_{k}$,
    if and only if
    the bounded complement of $\gamma_{\ell}$
    is a subset of that of $\gamma_{k}$.
   If $\gamma_k$ includes $\gamma_{\ell}$ and
    $\gamma_k\ne \gamma_{\ell}$,
    we write $\gamma_{k}> \gamma_{\ell}$
    or $\gamma_{\ell}< \gamma_{k}$.
 \end{definition}

Definitions \ref{def:covering} and \ref{def:inclusionRelation}
 yield a covering relation for Jordan curves.

 \begin{definition}[Covering of Jordan curves]
   \label{def:CoveringJordanCurves}
   Let ${\cal J}$ denote a poset of Jordan curves
    with inclusion as the partial order.
   We say 
    $\gamma_k$ \emph{cover}s $\gamma_{\ell}$ in ${\cal J}$
    and write `$\gamma_{k} \succ \gamma_{\ell}$'
    or `$\gamma_{\ell}\prec \gamma_{k}$'
    if $\gamma_{\ell}< \gamma_{k}$ and no elements $\gamma\in {\cal J}$
    satisfies $\gamma_{\ell}< \gamma<\gamma_{k}$.
 \end{definition}

 \begin{figure}
   \centering
   \includegraphics[width=0.2\linewidth]{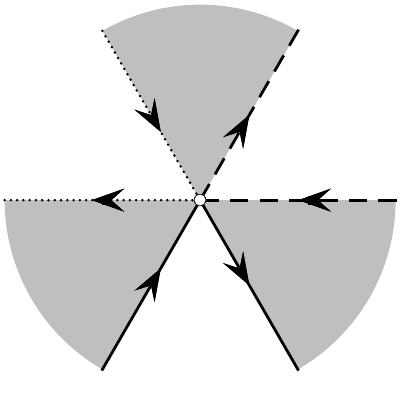}
   \caption{Decomposing connected $\partial {\cal Y}$
      into a set of oriented Jordan curves.
     The connected Yin set ${\cal Y}$ is represented
      by shaded regions.
     The Jordan curves determined
      by the choice shown in Figure \ref{fig:edgePairing}
      can be uniquely oriented by requiring that
      ${\cal Y}$ always lies at the left of each curve.
   }
 \label{fig:edgePairingDirected}
\end{figure}

 \begin{figure}
   \centering
   \subfigure[two incomparable negatively oriented Jordan curves]{
     \includegraphics[width=0.3\linewidth]{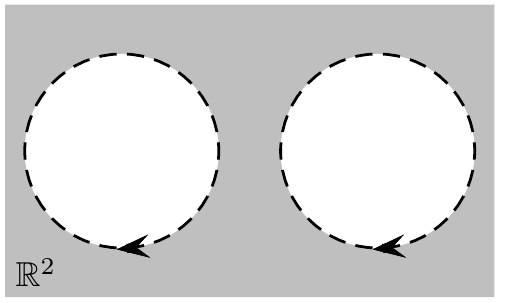}
   }
   \hfill
   \subfigure[two comparable negatively oriented Jordan curves]{
     \includegraphics[width=0.3\linewidth]{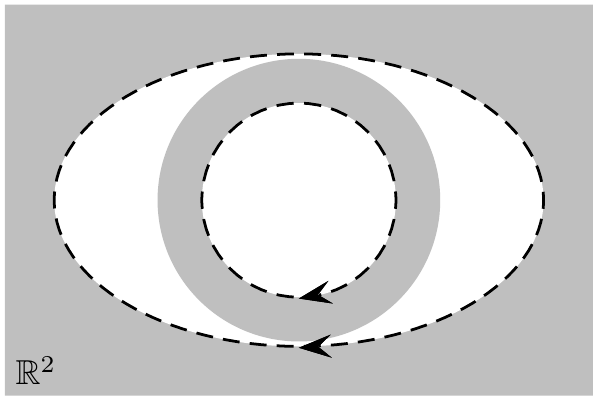}
   }

   \subfigure[a positively oriented Jordan curve covers a
   negatively oriented one]{
     \includegraphics[width=0.3\linewidth]{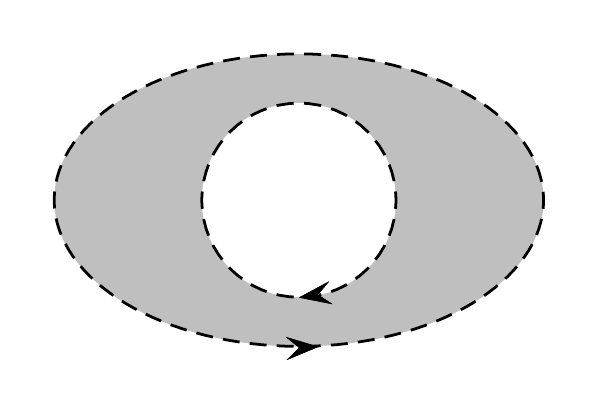}
   }
   \hfill
   \subfigure[a negatively oriented Jordan curve covers a
   positively oriented one]{
     \includegraphics[width=0.3\linewidth]{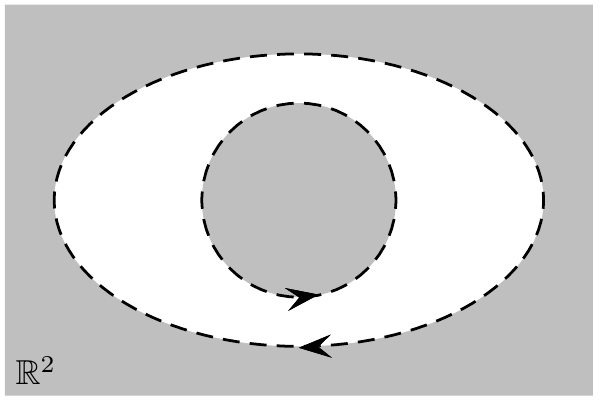}
   }
   \hfill
   \subfigure[two incomparable Jordan curves with different orientations]{
     \includegraphics[width=0.3\linewidth]{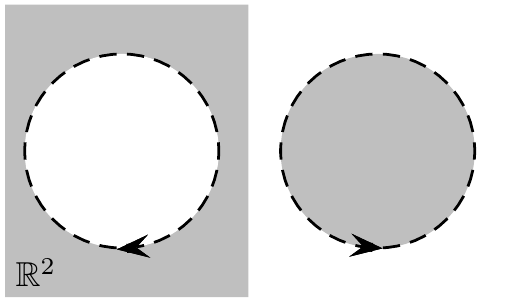}
   }

   \subfigure[two incomparable positively oriented Jordan curves]{
     \includegraphics[width=0.3\linewidth]{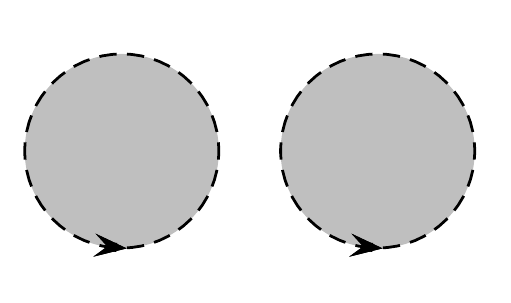}
   }
   \hfill
   \subfigure[two comparable positively oriented Jordan curves]{
     \includegraphics[width=0.3\linewidth]{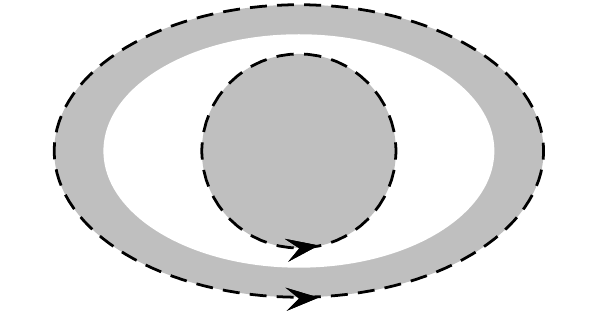}
   }
   \caption{Enumerating all cases of
     two oriented almost disjoint Jordan curves
     $\gamma_1$ and $\gamma_2$
     with respect to their orientations and inclusion relations.
     This is useful in proving Theorem \ref{thm:uniqueCases}:
      if $\gamma_1$ and $\gamma_2$
      are in the unique decomposition
      of the boundary of a connected Yin set ${\cal Y}$
      and we require
      that ${\cal Y}$ always be at the left side of
      both $\gamma_1$ and $\gamma_2$,
      then only (a) and (c) are valid, 
      because (b), (e), and (g) contradict
      the fact of both $\gamma_1$ and $\gamma_2$
      are part of the boundary of ${\cal Y}$
      and (d) and (f) contradict the condition
      of ${\cal Y}$ being connected.
   }
 \label{fig:enumerationOfTwoJCs}
\end{figure}

Now we state the most important result of this subsection.

 \begin{theorem}
   \label{thm:uniqueCases}
   Suppose a Yin set ${\cal Y}\ne \emptyset, \mathbb{R}^2$
    is connected.
   Then the Jordan curves as the unique decomposition
    of $\partial {\cal Y}$,
    given by the method shown in Figure \ref{fig:edgePairing}, 
    can further be uniquely oriented such that
    \begin{equation}
      \label{eq:intersectInteriors}
      {\cal Y} = \bigcap_{\gamma_j\in {\cal J}_{\partial {\cal Y}}} \Int(\gamma_j).
    \end{equation}
    ${\cal J}_{\partial {\cal Y}}$,
     the set of oriented boundary Jordan curves of ${\cal Y}$,
     must be one of the two types,
     \begin{equation}
       \label{eq:decomTypes}
       \renewcommand{\arraystretch}{1.2}
       \left\{
         \begin{array}{ll}
           {\cal J}^-
           =\{\gamma^-_1, \gamma^-_2, \ldots, \gamma^-_{n_-}\},
           & n_-\ge 1,
           \\
           {\cal J}^+
           =\{\gamma^+,\gamma^-_1, \gamma^-_2, \ldots, \gamma^-_{n_-}\},
           & n_-\ge 0,
         \end{array}
         \right.
     \end{equation}
     where all $\gamma^-_j$'s are negatively oriented,
     mutally incomparable with respect to inclusion.
     For ${\cal J}^+$, we also have
     \begin{equation}
       \label{eq:negCoveredByPos}
       \forall j=1,2,\ldots,n_-,\ \ 
       \gamma_j^- \prec \gamma^+.
     \end{equation}
 \end{theorem}
 \begin{proof}
   As in the proof of Theorem \ref{thm:uniqueRep},
    we construct a multigraph $G_{\partial {\cal Y}}$
    from $\partial {\cal Y}$.
   $G_{\partial {\cal Y}}$ is further made a directed multigraph
    via orienting $\partial {\cal Y}$ so that ${\cal Y}$
    always lies at the left side of any oriented Jordan curve.
   As shown in Figure \ref{fig:edgePairingDirected},
    the indegree and outdegree of any vertex
    in $G_{\partial {\cal Y}}$ equals.
   By Theorem \ref{thm:VeblenMultigraphDirected},
    there exists at least one directed cycle decomposition
    of $\partial {\cal Y}$.
   Then the uniqueness of the directed cycle decomposition
    follows from the uniqueness of orienting the Jordan curves
    in Figure \ref{fig:edgePairingDirected}.

   Consider the case ${\cal J}_{\partial {\cal Y}}={\cal J}^-$.
   For $n_->1$, suppose that
    $\gamma_1^-$ and $\gamma_2^-$ is comparable,
    as shown in Figure \ref{fig:enumerationOfTwoJCs}(b).
   According to the orientation,
    ${\cal Y}$ must lie at the left side of both $\gamma_1^-$
    and $\gamma_2^-$,
    but this is impossible because both $\gamma_1^-$
    and $\gamma_2^-$ are part of the boundary of ${\cal Y}$.
   Consequently,
    (\ref{eq:intersectInteriors}) follows
    from Definition \ref{def:interior};
    see Figure \ref{fig:enumerationOfTwoJCs}(a).

   Consider the case ${\cal J}_{\partial {\cal Y}}={\cal J}^+$.
   If $n_-=0$, then (\ref{eq:negCoveredByPos}) holds vacuously
    and (\ref{eq:intersectInteriors}) holds trivially
    from Definition \ref{def:interior}.
   For $n_->0$,
    suppose (\ref{eq:negCoveredByPos}) did not hold
    for a negatively oriented Jordan curve $\gamma_1^-$.
   Then the almost disjointness implies that
    either $\gamma^+\prec \gamma_1^-$
    or they are not comparable.
   Suppose the former case holds, 
    a path from one point at the left of $\gamma^+$
    to another point at the left of $\gamma_1^-$
    must contain some points not in ${\cal Y}$,
    which contradicts the condition of ${\cal Y}$ being connected;
    see Figure \ref{fig:enumerationOfTwoJCs} (d).
   The latter case does not hold either
    because it contradicts the fact
    that $\gamma^+$ is part of the boundary of ${\cal Y}$;
    see Figure \ref{fig:enumerationOfTwoJCs} (e).
   Hence (\ref{eq:negCoveredByPos}) must hold.
   By arguments in the previous paragraph,
    the negatively oriented Jordan curves
    must also be pairwise incomparable.
   Therefore, (\ref{eq:intersectInteriors}) follows
    from Definition \ref{def:interior};
    see Figure \ref{fig:enumerationOfTwoJCs} (c).

   Suppose ${\cal J}_{\partial {\cal Y}}$ contains
    two positively oriented Jordan curves $\gamma_1,\gamma_2$.
   Then their almost disjointness implies
    that $\Int(\gamma_1)$ is either in the unbounded complement
    or the bounded complement of $\gamma_2$.
   By similar arguments, the former contradicts
    the fact of ${\cal Y}$ being connected,
    as in Figure \ref{fig:enumerationOfTwoJCs} (f),
    and the latter contradicts
    the fact of both $\gamma_1$ and $\gamma_2$
    are part of the boundary of ${\cal Y}$,
    as in Figure \ref{fig:enumerationOfTwoJCs} (f).
   Hence ${\cal J}_{\partial {\cal Y}}$ contains
    at most one positively oriented Jordan curve.
   This completes the proof.
 \end{proof}

\begin{corollary}
  \label{coro:JordanCurveRepGeneral}
  Each Yin set ${\cal Y}\ne \emptyset, \mathbb{R}^2$
   can be uniquely expressed as
   \begin{equation}
     \label{eq:JordanCurveRepGeneral}
     {\cal Y}= {\bigcup}_{j}^{\perp\perp} \bigcap_i  \mathrm{int}\left(\gamma_{j,i}\right),
   \end{equation}
   where $j$ is the index of connected components of ${\cal Y}$
   and $\gamma_{j,i}$'s are oriented Jordan curves
   that are pairwise almost disjoint.
\end{corollary}
\begin{proof}
  The conclusion follows from
  applying Theorem \ref{thm:uniqueCases} to each connected component
  of ${\cal Y}$.
\end{proof}

We illustrate (\ref{eq:intersectInteriors})
 and (\ref{eq:JordanCurveRepGeneral})
 by the two distinct types of Yin sets
 in Figure \ref{fig:OrientYinSets}.

 \begin{figure}
   \centering
   \subfigure[a connected Yin set]{
     \includegraphics[width=0.45\linewidth]{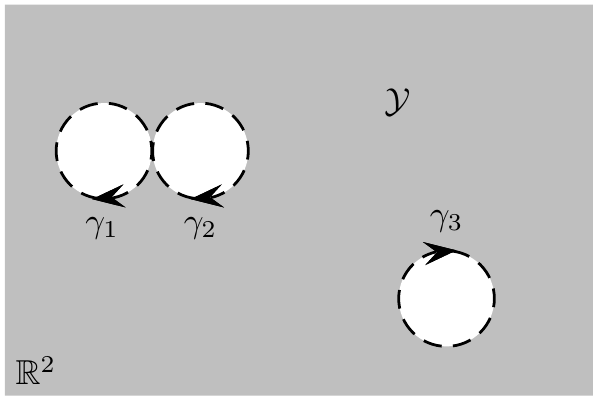}
   }
   \hfill
   \subfigure[a Yin set with four connected components]{
     \includegraphics[width=0.45\linewidth]{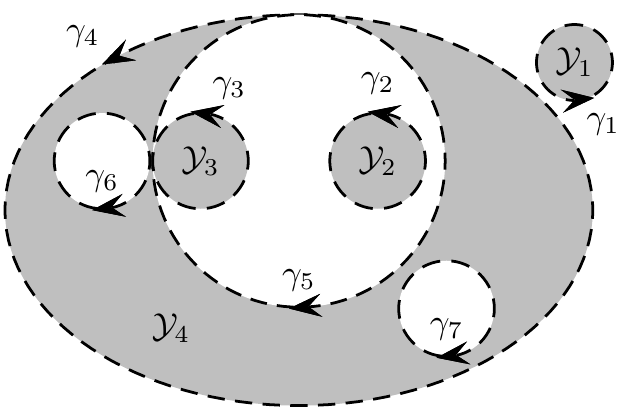}
   }
   \caption{Orienting boundary Jordan curves
     of the Yin sets in Figure \ref{fig:YinSets}
     as illustrations of the two types of connected Yin sets
      classified in Theorem \ref{thm:uniqueCases}.
     In subplot (a), ${\cal Y}=\bigcap_{j=1}^3\Int(\gamma_j)$;
     in subplot (b), ${\cal Y}=\bigcup_{i=1}^4{\cal Y}_i
     = \Int(\gamma_1)\cup\Int(\gamma_2)\cup\Int(\gamma_3)\cup
     \left[\bigcap_{j=4}^7\Int(\gamma_j)\right]$.
     By Theorem \ref{thm:uniqueCases},
      the boundaries of the connected Yin sets ${\cal Y}$ in (a)
      and ${\cal Y}_4$ in (b)
      are of the types ${\cal J}^-$ and ${\cal J}^+$, respectively.
   }
 \label{fig:OrientYinSets}
\end{figure}

By results on the global topology,
 it is straightforward
 to identify the Betti numbers of a Yin set
 with the numbers of oriented Jordan curves in its representation.

\begin{corollary}
  \label{coro:BettiNumbers}
  The number of holes in a connected Yin set
   is the number of negatively oriented Jordan curves
   in the unique expression (\ref{eq:intersectInteriors}).
  The number of connected components
   in a bounded Yin set is the number
   of positively oriented Jordan curves
   in the unique expression (\ref{eq:JordanCurveRepGeneral}).
\end{corollary}

This simple result is due to the topological stratification
 of the Yin space
 and the natural correspondence of 
 holes to negatively oriented Jordan curves.

\subsection{$\mathbb{J}$: 
representing Yin sets via realizable spadjors}
\label{sec:spadjors}

Our starting point is the following acronym.

 \begin{definition}
   \label{def:spadjor}
   A \emph{spadjor} 
    is a nonempty set of pairwise almost disjoint Jordan curves
    with orientations.
 \end{definition}

Not all spadjors are useful for representing Yin sets.
The global topology of Yin sets in Theorem \ref{thm:uniqueCases}
 and Corollary \ref{coro:JordanCurveRepGeneral}
 naturally suggests that we should limit our attention
 to certain types of spadjors.

 \begin{definition}
   \label{def:atomSpadjors}
   An \emph{atom spadjor} is a spadjor
    that consists of at most one positively oriented Jordan curve
    $\gamma^+$
    and a finite number of negatively oriented
    Jordan curves $\gamma^-_{1}$, $\gamma^-_{2}$, $\ldots$,
    $\gamma^-_{n_-}$ such that
    \begin{enumerate}[(a)]
    \item $\gamma^-_{j}$'s are pairwise incomparable
      with respect to inclusion,
    \item $\gamma^-_{\ell}\prec \gamma^+$ 
      for each $\ell=1,2,\ldots,n_-$.
    \end{enumerate}
 \end{definition}

Since a spadjor cannot be an empty set,
 $n_- = 0$ implies the presence of $\gamma^+$ and 
 the absence of $\gamma^+$ implies $n_->0$.
By definition,
 a spadjor is in the form of either ${\cal J}^-$
 or ${\cal J}^+$ in (\ref{eq:decomTypes}).

The following {boundary-to-interior map} $\rho$
 assigns to each atom spadjor a connected Yin set,
 \begin{equation}
   \label{eq:boundaryToInteriorMapSpadjor}
   \rho({\cal J}_k) := {\bigcap}_{\gamma_{i}\in {\cal J}_k} \Int(\gamma_{i}).
 \end{equation}    

\begin{definition}
  \label{def:realizableSpadjor}
 A \emph{realizable spadjor} ${\cal J}=\cup_k{\cal J}_k$
  is the union of a finite number
  of atom spadjors
  such that ${\cal J}_i$ and ${\cal J}_j$ being distinct implies
  $\rho({\cal J}_i)\cap \rho({\cal J}_j) =\emptyset$.
\end{definition}

Intuitively,
 an atom spadjor represents a connected Yin set
 while a realizable spadjor may represent 
 a Yin set with multiple connected components.
For example,
 the Yin set in Figure \ref{fig:OrientYinSets}(b)
 can be represented 
 by the realizable spadjor 
 \begin{equation}
   \label{eq:YinSetRep2}
   {\cal J}=\left\{ \gamma^+_1, \gamma_4^+, \gamma_6^-, \gamma_5^-,
     \gamma_7^-, \gamma_2^+, \gamma_3^+ \right\}.
 \end{equation}

Lemma \ref{lem:RS2AS} concerns recovering
 the atom spadjors in a realizable spadjor.

 \begin{lemma}
   \label{lem:RS2AS}
   Any realizable spajor ${\cal J}$ can be \emph{uniquely} expressed as
   \begin{equation}
     \label{eq:realizableSpadjorDecomposition}
     {\cal J} = \cup_{i=1}^n {\cal J}^{+}_i \cup {\cal J}^{-},
   \end{equation}
   where the ${\cal J}^{+}_i$'s and the ${\cal J}^{-}$
    are extracted from ${\cal J}$ as follows.
    \begin{enumerate}[({R2A}-a)]
    \item For each positively oriented Jordan curve $\gamma^+_i\in {\cal J}$,
      form an atom spadjor ${\cal J}^+$ by adding $\gamma_i^+$
      and all negatively oriented Jordan curves
      covered by $\gamma_i^+$.
    \item If there are negatively oriented Jordan curves left in ${\cal J}$,
      group them into an atom spadjor ${\cal J}^-$.
    \end{enumerate}
 \end{lemma}
 \begin{proof}
  The existence of the expression in
   (\ref{eq:realizableSpadjorDecomposition})
   follows from Definition \ref{def:realizableSpadjor}
   and Theorem \ref{thm:uniqueCases}.
  As for the uniqueness,
   we first note that at most one atom spadjor in the form of ${\cal J}^-$
   can be extracted from any realizable spadjor;
   otherwise it would contradict the condition
   $\rho({\cal J}_i)\cap \rho({\cal J}_j) =\emptyset$
   in Definition \ref{def:realizableSpadjor}.
  Second,
   ${\cal J}^-\ne \emptyset$ and $n>0$ imply
   that the positively oriented Jordan curve
   $\gamma^+_k$ of any ${\cal J}^+_k\subset {\cal J}$
   is covered by a negatively oriented element $\gamma_i^-\in {\cal J}^-$.
  Third,
   any negatively oriented Jordan curve
   not belonging to ${\cal J}^-$
   can be \emph{covered} by at most one ${\cal J}^+_k$.
 \end{proof}

The partition of a realizable spadjor ${\cal J}$ into atom spadjors
 is best illustrated by
 the Hasse diagram of the poset ${\cal J}$ with respect to inclusion,
 c.f. Figure \ref{fig:HassDiagram}.

 \begin{figure}
   \centering
   \includegraphics[width=0.45\linewidth]{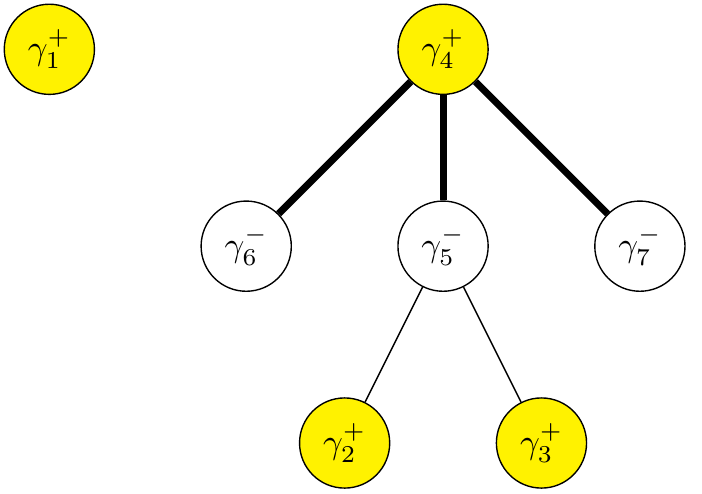}
   \caption{The Hasse diagram
     for the realizable spadjor ${\cal J}$ in (\ref{eq:YinSetRep2})
     that represents the Yin set in Figure \ref{fig:OrientYinSets}(b).
     The partial order is the  ``inclusion'' relation
      as in Definition \ref{def:inclusionRelation}.
     A shaded circle represents a positively oriented Jordan curve
      while an unshaded circle a negatively oriented Jordan curve.
     The partition of ${\cal J}$ into atom spadjors
      consists of two steps:
      (a) form an atom spadjor from each shaded circle
      and its immediate children (if there is any)
      and (b) if any unshaded circles remain,
      group them into ${\cal J}^-$.
     These two steps correspond to (R2A-a,b) in Lemma \ref{lem:RS2AS}.
   }
 \label{fig:HassDiagram}
\end{figure}

 \begin{definition}
   \label{def:JordanSpace}
   The \emph{Jordan space} is the set
   \begin{equation}
     \label{eq:JordanSpace}
     \mathbb{J}:=\{\Zero,\One\} \cup \{\cal J\}
   \end{equation}
    where $\{\cal J\}$ denotes
    the set of all realizable spadjors
    and $\Zero$, $\One$ are two symbols satisfying
    \begin{equation}
      \label{eq:ZeroAndOne}
      \rho(\Zero) := \emptyset,\qquad
      \rho(\One) := \mathbb{R}^2.
    \end{equation}
 \end{definition}

One can interpret $\Zero$ and $\One$
 as atom spadjors consisting of a single oriented Jordan curve
 with infinitesimal diameter such that
 its interior goes to $\emptyset$ and $\mathbb{R}^2$,
 respectively.
 
It now makes sense to extend the definition
 of the boundary-to-interior map in (\ref{eq:boundaryToInteriorMapSpadjor})
 and (\ref{eq:ZeroAndOne})
 to the Jordan space.

 \begin{definition}
   \label{def:boundaryToInteriorMap}
   The \emph{boundary-to-interior map}
    $\rho: \mathbb{J}\rightarrow \mathbb{Y}$
    associates a Yin set 
    with each element in the Jordan space as
    \begin{equation}
      \label{eq:boundaryToInteriorMap}
      \rho({\cal J}) :=\left\{
        \begin{array}{ll}
          \emptyset & \textrm{if } {\cal J}=\Zero;
          \\
          \mathbb{R}^2 & \textrm{if } {\cal J}=\One;
          \\
         {\bigcup}_{{\cal J}_j\subset {\cal J}}^{\perp\perp}
          {\bigcap}_{\gamma_{i}\in {\cal J}_j} \Int(\gamma_{i})
          & \textrm{otherwise},
                               
        \end{array}
        \right.
    \end{equation}
    where the ${\cal J}_i$'s are atom spadjors
    extracted from ${\cal J}$ as in Lemma \ref{lem:RS2AS}.
 \end{definition}

 \begin{theorem}
   \label{thm:b2iMapIsBijective}
   The boundary-to-interior map in Definition \ref{def:boundaryToInteriorMap}
    is bijective.
 \end{theorem}
 \begin{proof}
   For the Yin sets $\emptyset$ and $\mathbb{R}^2$,
    (\ref{eq:ZeroAndOne}) states that $\Zero$ and $\One$
    are their preimages.
   For any other Yin set ${\cal Y}\ne\emptyset,\mathbb{R}^2$,
    we can uniquely decompose it as 
${\cal Y} = {\bigcup}_{{\cal Y}_i\subseteq {\cal Y}}^{\perp\perp} {\cal Y}_i$,
    where the connected components ${\cal Y}_i$'s are pairwise disjoint.
    By Theorem \ref{thm:uniqueCases},
     each ${\cal Y}_i$ is uniquely expressed as
     the intersection of interiors of a number of oriented Jordan curves.
    Hence $\rho$ is surjective.
    Also, $\rho$ is injective because of the uniqueness
    in Lemma \ref{lem:RS2AS}.
 \end{proof}

 \begin{corollary}
   \label{coro:uniqueRepOfYinSets}
   A Yin set is uniquely represented
   by a realizable spadjor.
 \end{corollary}
 \begin{proof}
   This follows from Theorem \ref{thm:b2iMapIsBijective}
    and Corollary \ref{coro:JordanCurveRepGeneral}.
 \end{proof}

 \begin{figure}
   \centering
   \subfigure[a panda modeled as a Yin set ${\cal P}$]{
     \includegraphics[width=0.46\linewidth]{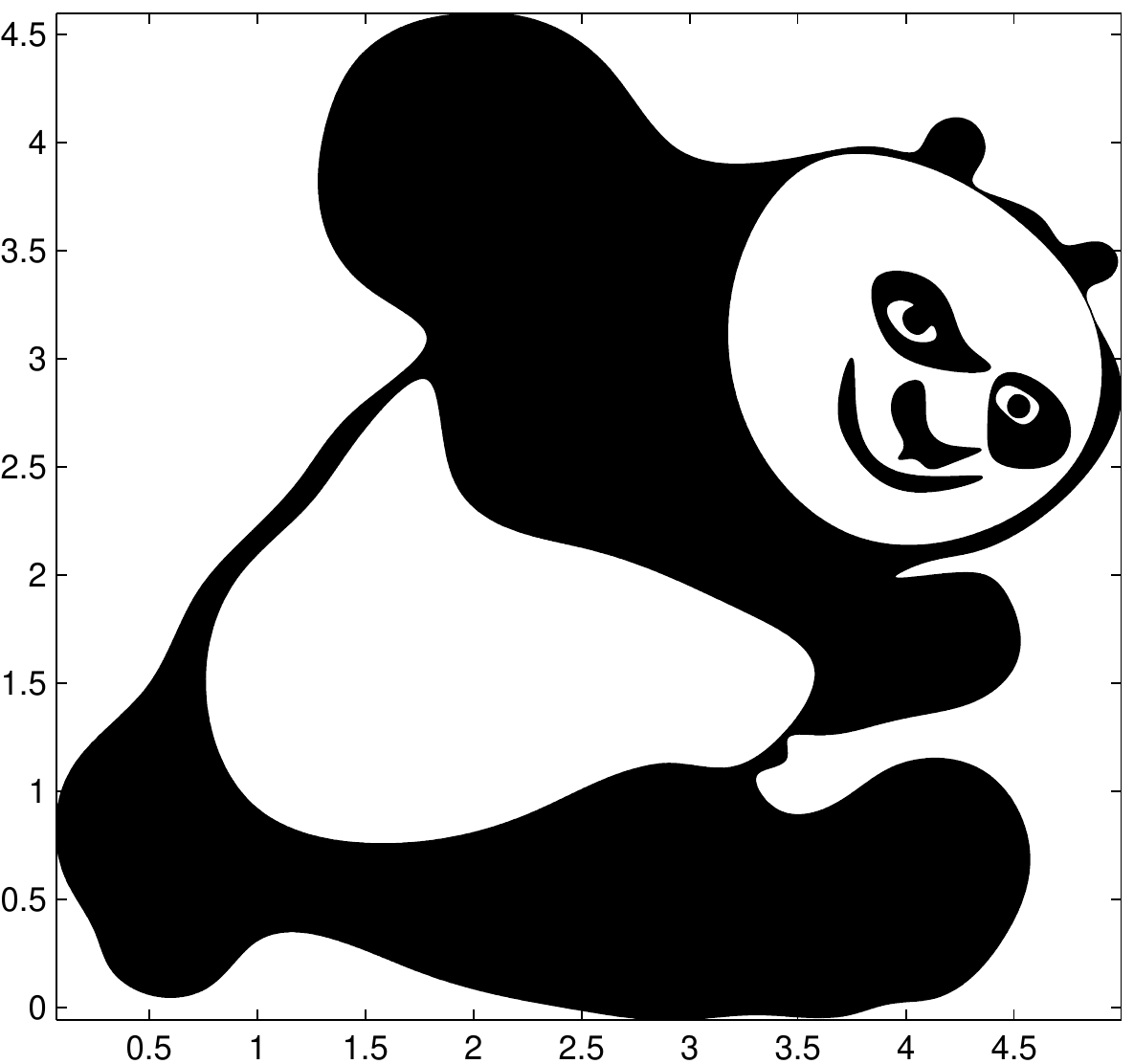}
   }
   \hfill
   \subfigure[the unique representation of ${\cal P}$ as
   a realizable spadjor ${\cal J}$]{
     \includegraphics[width=0.47\linewidth]{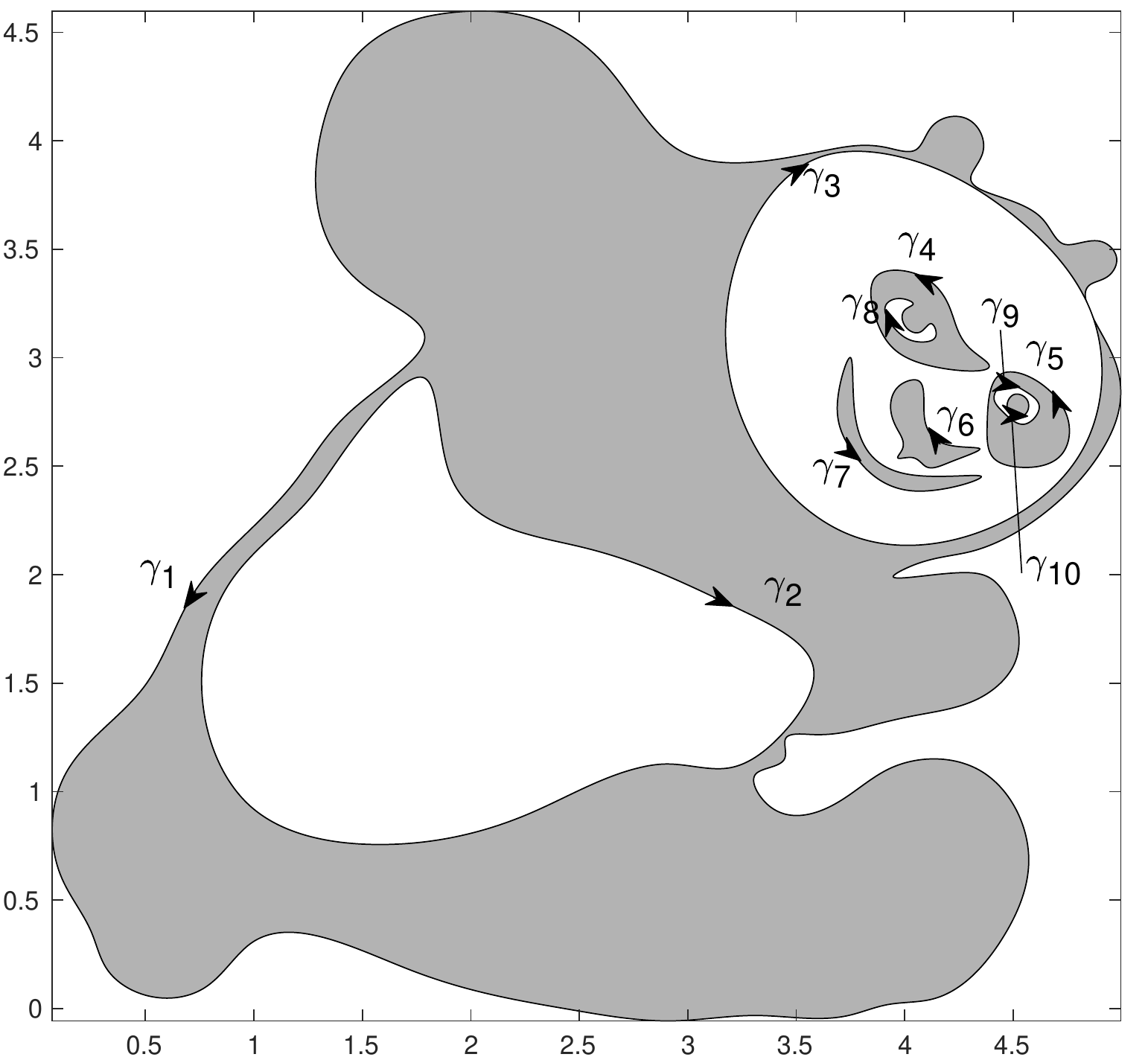}
   }
   \caption{A Yin set with complex topology and geometry.
     In subplot (b), the Jordan curves
      $\gamma_1$, $\gamma_4$, $\gamma_5$,
      $\gamma_6$, $\gamma_7$, and $\gamma_{10}$
      are positively oriented
      while the others are negatively oriented.
     The realizable spadjor is
      ${\cal J}=\cup_{k=1}^6{\cal J}_k$, 
      with the atom spadjors as 
      ${\cal J}_1 = \{\gamma_1, \gamma_2, \gamma_3\}$,
      ${\cal J}_2 = \{\gamma_4, \gamma_8\}$,
      ${\cal J}_3 = \{\gamma_5, \gamma_9\}$,
      ${\cal J}_4 = \{\gamma_6\}$,
      ${\cal J}_5 = \{\gamma_7\}$,
      and ${\cal J}_6 = \{\gamma_{10}\}$;
      all atom spadjors are of the type ${\cal J}^+$
      in (\ref{eq:decomTypes}). 
     The panda is uniquely expressed as
      ${\cal P}=\cup_{k=1}^6{\cal P}_k=\rho({\cal J})$
      with each connected component
      as ${\cal P}_k=\rho({\cal J}_k)$.
     The picture in subplot (a) is a raster image
      while the curves in subplot (b)
      are cubic splines fit through a total of 120 points.
     This small amount of points demonstrates
      the efficiency of realizable spadjor
      in representing complex topology and geometry.
   }
 \label{fig:panda}
\end{figure}

We sum up this section by Figure \ref{fig:panda},
 where a physically meaningful region with complex topology
 is modeled by a fun Yin set,
 which is further uniquely represented by a realizable spadjor.



\section{The Boolean algebra on Yin sets}
\label{sec:boolean-algebra-yin}

After introducing the pasting map
 in Section \ref{sec:pastingMap},
 we define 
 in Sections \ref{sec:compl-op} and \ref{sec:meet-operation}
 the complementation and the meet operations on realizable spadjors
 to equip the Jordan space $\mathbb{J}$
 as a bounded distributive lattice.
Along the way,
 we show that these operations
 are counterparts to Boolean operations on the Yin space.
Our theory culminates
in Section \ref{sec:isomorphism}.
In Section \ref{sec:algor-impl}, 
 we discuss implementation issues
 and present a fun example of our Boolean algorithms on Yin sets.

\subsection{The pasting map of realizable spadjors}
\label{sec:pastingMap}

The following cutting map is trivial,
 but its inverse in Lemma \ref{lem:segmentationMapInverse} is not.
They are crucial for the complementation and the meet operations
 in Sections \ref{sec:compl-op} and \ref{sec:meet-operation}.

 \begin{definition}
   \label{def:segmentationMap}
   For a realizable spadjor ${\cal J}$,
    let $V$ denote a finite point set that
    contains all intersections of the Jordan curves
    in ${\cal J}$.
  The associated \emph{cutting map} or \emph{segmentation map} $S_V$ 
    assigns to ${\cal J}$ a set $E$ of 
    oriented paths obtained by 
    dividing the oriented Jordan curve in ${\cal J}$ at points in $V$.
  The set of (oriented) paths $E=S_V({\cal J})$
    is called a \emph{segmented realizable spadjor}.
 \end{definition}

 \begin{lemma}
   \label{lem:segmentationMapInverse}
   The realizable spadjor ${\cal J}=S_V^{-1}(E)$
   corresponding to a segmented realizable spadjor $E$
   can be uniquely constructed as follows.
   \begin{enumerate}[({S2R}-a)]
   \item Remove all self-loops in $E$ and insert them into ${\cal J}$.
   \item Start with a path $\beta_{\textrm{in}}\in E$
     and denote by $v\in V$ its ending point.
     If there exists only one path
     whose starting point is $v$, 
     call it $\beta_{\textrm{out}}$  and append it to $\beta_{\textrm{in}}$.
     Otherwise,
     set $\beta_{\textrm{out}}$ to be the edge
     that starts at $v$
     and of which the positive counter-clockwise
     angle $\angle \beta_{\textrm{out}} v \beta_{\textrm{in}}$
     is the smallest.
     Repeat the above conditional to grow the path until
     it becomes a loop $\gamma_1$,
     and remove from $E$ all paths that constitute $\gamma_1$.
   \item If $\gamma_1$ is a Jordan curve,
     add it into ${\cal J}$;
     otherwise divide $\gamma_1$ into Jordan curves and/or self-loops 
     and add them into ${\cal J}$.
   \item Repeat (S2R-b,c) to add other Jordan curves
     into ${\cal J}$ until $E$ becomes empty.
   \end{enumerate}
 \end{lemma}
 \begin{proof}
   First we note that Theorem \ref{thm:uniqueRep}
    is not applicable here
    because the Yin set $\rho({\cal J})$ may be disconnected.
   By Definition \ref{def:interior}, 
    the Yin set $\rho({\cal J})$ always lies at the left
    of any oriented path $\beta\in E$,
    the steps (S2R-a,b) are sufficient to fulfill this invariant.
   The choice for growing the path in (S2R-b) also avoids
    potential proper intersections of Jordan curves in ${\cal J}$; 
    see Figure \ref{fig:edgePairingProof}(a). 
   However, as suggested by Figure \ref{fig:edgePairingProof}(b)
    and the proof of Theorem \ref{thm:uniqueRep},
    the loop $\gamma_1$ might not be a single Jordan curve,
    hence we need to divide it into Jordan curves and/or self-loops
     in \mbox{(S2R-c)};
     see Figure \ref{fig:pastingMap}.
    The uniqueness of the constructed realizable spadjor
     follows from the uniqueness
     in Corollary \ref{coro:uniqueRepOfYinSets}.
 \end{proof}
    
  \begin{figure}
    \centering
    \subfigure[Two components and one improper intersection]{
    \includegraphics[width=0.25\linewidth]{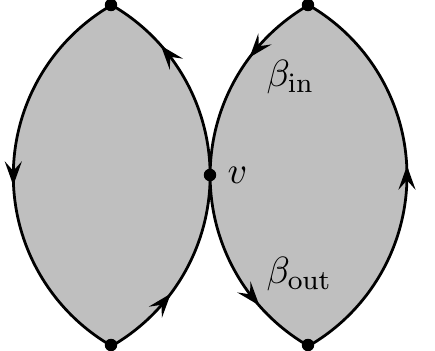}}
    \hfill
    \subfigure[Two components and two improper intersections]{
    \includegraphics[width=0.31\linewidth]{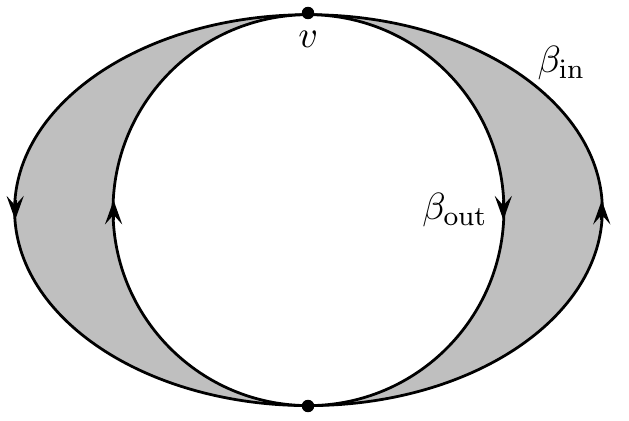}}
    \hfill
    \subfigure[One component and two improper intersections]{
    \includegraphics[width=0.3\linewidth]{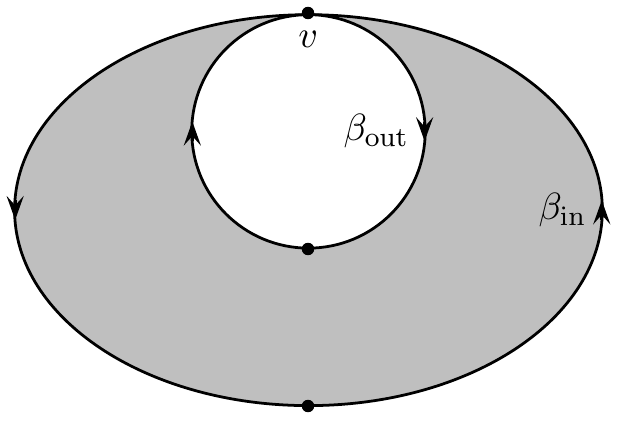}}
    \caption{Illustrating key steps (S2R-b,c) of the pasting map
      in Lemma \ref{lem:segmentationMapInverse}.
      The shaded region represents a Yin set,
       whose boundary consists of two Jordan curves
       with an improper intersection at $v$.
      In each subplot, the solid dots represent points of $V$
       while the directed paths constitute $E$.
      Starting from $\beta_{\textrm{in}}$,
       we pick $\beta_{\textrm{out}}$ to grow the starting path
       because $\angle \beta_{\textrm{out}} v \beta_{\textrm{in}}$
       is the smallest counterclockwise angle
       among those of the two out-edges;
       this condition in (S2R-b)
       is different from that in Figure \ref{fig:edgePairing}.
      In subplot (c), the loop $\gamma_1$ resulting from (S2R-b)
       consumes all paths.
      Hence in (S2R-c) we divide it into two Jordan loops
       to fulfill the representation invariant of realizable spadjors.
    }
    \label{fig:pastingMap}
  \end{figure}

The procedures in Lemma \ref{lem:segmentationMapInverse}
 define the inverse of the cutting map,
 to which we refer as the \emph{pasting map of realizable spadjors}.

\subsection{Complementation: a unitary operation on $\mathbb{J}$}
\label{sec:compl-op}

The following is an easy result on the local topology of regular sets.

 \begin{corollary}
   \label{coro:brdPtOfRegularSets}
  Suppose ${\cal Y}$ is a regular open or regular closed set.
  Then a point $p\in \mathbb{R}^2$
   is a boundary point of ${\cal Y}$
   if and only if, for any sufficiently small $r>0$,
   the open ball ${\cal N}_r(p)$ centered at $p$ with radius $r$
   contains both points in ${\cal Y}$ and ${\cal Y}^{\perp}$.
 \end{corollary}
 \begin{proof}
   The necessity follows directly from Lemma
    \ref{lem:localTopologyOfBdPtOfYinSet} (c),
    we only prove the sufficiency.
   If $r$ is sufficiently small,
    there are only seven cases for the type of points
    contained in ${\cal N}_r(p)$:
    (i) ${\cal Y}$, (ii) ${\cal Y}^{\perp}$,
    (iii) $\partial {\cal Y}$,
    (iv) ${\cal Y}$ and ${\cal Y}^{\perp}$,
    (v) ${\cal Y}$ and $\partial {\cal Y}$,
    (vi) ${\cal Y}^{\perp}$ and $\partial {\cal Y}$,
    and (vii) all three sets.
   Because of the regularity,
    cases (iii) to (vi) are impossible.
   By definitions in Section \ref{sec:regularSets},
    (i) implies an interior point and
    (ii) implies an exterior point.
   Hence (vii) must imply a boundary point.
   In other words,
    the presence of ${\cal Y}$ and ${\cal Y}^{\perp}$
    in ${\cal N}_r(p)$
    dictates that of $\partial {\cal Y}$ in ${\cal N}_r(p)$.
 \end{proof}

Corollary \ref{coro:brdPtOfRegularSets}
 and Lemma \ref{lem:segmentationMapInverse} motivate
 our complementation operation on $\mathbb{J}$.

 \begin{definition}
   \label{def:CJJ}
   The \emph{complementation operation} $':\mathbb{J}\rightarrow\mathbb{J}$
    is defined as
    \begin{equation}
      \label{eq:CJJ}
      {\cal J}' := \left\{
      \begin{array}{ll}
        \One & \textrm{if } {\cal J}=\Zero;
        \\
        \Zero & \textrm{if } {\cal J}=\One;
        \\
        \left( S_V^{-1} \circ R \circ S_V \right) {\cal J}
        & \textrm{otherwise}, 
      \end{array}
      \right.
    \end{equation}
    where $V$ is the set of improper intersections of
     Jordan curves in ${\cal J}$
     and the orientation-reversing map $R$
     reverses the orientation of each path in
     the segmented realizable spadjor $S_V({\cal J})$.
 \end{definition}

 \begin{lemma}
   \label{lem:complementationHomo}
   The complementation operation in Definition \ref{def:CJJ}
   satisfies 
   \begin{equation}
     \label{eq:complementationHomo}
     \forall {\cal J} \in \mathbb{J},
     \qquad \rho({\cal J}')
     = (\rho({\cal J}))^{\perp}.
   \end{equation}
 \end{lemma}
 \begin{proof}
   By definition,
    a regular set ${\cal Y}$ satisfies
    $({\cal Y}^{\perp})^{\perp}={\cal Y}$,
    which, together with Corollary \ref{coro:brdPtOfRegularSets},
    imply that ${\cal Y}$ and ${\cal Y}^{\perp}$
    have exactly the same boundary.
   By Corollary \ref{coro:JordanCurveRepGeneral}
    and Definition \ref{def:interior},
    the realizable spadjors
    representing ${\cal Y}$ and ${\cal Y}^{\perp}$
    are the same set of Jordan curves except that
    each pair of corresponding Jordan curves
    has different orientations,
    which justifies the necessity of the orientation-reversing map $R$.
   More precisely,
    the conjugate of $R$ by the cutting map $S_V$ is needed
    here because multiple improper intersections
    of two Jordan curves 
    may dictate that the paths
    constituting oriented Jordan curves in ${\cal J}$
    be reorganized in order to represent ${\cal Y}^{\perp}$ properly.
   For example,
    the calculation of ${\cal Y}^{\perp}$ for the Yin set ${\cal Y}$
    in Figure \ref{fig:pastingMap} (b)
    involves not only reversing the orientation of each path,
    but also reorganizing these paths into different Jordan curves.
 \end{proof}

\subsection{Meet: a binary operation on $\mathbb{J}$}
\label{sec:meet-operation}

To define the meet of
 two realizable spadjors ${\cal J}$ and ${\cal K}$,
 we cut them by $S_V$,
 select those paths that are on the boundary of the Yin set 
 $\rho({\cal J}) \cap \rho({\cal K})$,
 and paste the set of selected paths by $S_V^{-1}$ to form the result.
Lemma \ref{lem:bryEdgeRequirement}
 and Corollary \ref{coro:bryEdgeRequirementForest}
 tell us which paths we should choose.

 \begin{lemma}
   \label{lem:bryEdgeRequirement}
   Denote 
   ${\cal Y}:=\mathrm{int}(\sigma_1)\cap \mathrm{int}(\sigma_2)$
   where $\sigma_1$ and $\sigma_2$ are two oriented Jordan curves.
   For a curve $\beta$ satisfying 
    $\beta\subseteq \sigma_1$
    and $\beta\subset \mathbb{R}^2\setminus\sigma_2$,
    we have $\beta\subseteq \partial {\cal Y}$ 
    if and only if $\beta\subset \mathrm{int}(\sigma_2)$.
   For a curve $\beta$ satisfying 
     $\beta\subseteq \sigma_1\cap \sigma_2$,
     we have $\beta\subseteq \partial {\cal Y}$ 
     if and only if 
     the direction of $\beta$
     induced from the orientation of $\sigma_1$
     is the same as that
     induced from the orientation of $\sigma_2$.
 \end{lemma}
 \begin{proof}
   We only prove the first statement since the second one
    can be shown similarly.
   By Definition \ref{def:interior}, 
    $\mathrm{int}(\sigma_1)$
    and $\mathrm{int}(\sigma_2)$ are both Yin sets.
   It follows from Theorem \ref{thm:YinSetsFormABooleanAlgebra} 
    that ${\cal Y}$ is also a Yin set.
   By Corollary \ref{coro:brdPtOfRegularSets},
    it suffices to show that a small open ball ${\cal N}_r(p)$
    centered at $p\in \beta$
    contains both points in ${\cal Y}$ and ${\cal Y}^{\perp}$
    if and only if $\beta\subset \mathrm{int}(\sigma_2)$.
   As shown in Figure \ref{fig:segmentOnBdry}, 
    $p\in \beta$ and $\beta\subseteq \sigma_1$
    imply that ${\cal N}_r(p)\cap \mathrm{int}(\sigma_1) \ne \emptyset$;
    then the condition $\beta\subset \mathrm{int}(\sigma_2)$
    implies ${\cal N}_r(p)\cap {\cal Y} \ne \emptyset$.
   DeMorgan's law (\ref{eq:DeMorgansLaws}) yields 
   \begin{displaymath}
     {\cal Y}^{\perp} = \left[ \mathrm{int}(\sigma_1)\cap
       \mathrm{int}(\sigma_2)\right]^{\perp}
     = \mathrm{int}(\sigma_1)^{\perp}\cup^{\perp\perp}
     \mathrm{int}(\sigma_2)^{\perp}
     \ne \emptyset,
   \end{displaymath}
    which, together with $\beta\subseteq \sigma_1$,
    imply ${\cal N}_r(p)\cap {\cal Y}^{\perp} \ne \emptyset$.
   By Corollary \ref{coro:brdPtOfRegularSets},
    we have $\beta\subset \partial {\cal Y}$.
   Conversely,
    $\beta\not\subset \mathrm{int}(\sigma_2)$,
    $\beta\subset \mathbb{R}^2\setminus\sigma_2$, and 
    $\beta\subseteq \sigma_1$ imply that
     ${\cal N}_r(q)\cap {\cal Y} = \emptyset$
     for all $q\in \beta$.
   These arguments are illustrated
    in Figure \ref{fig:segmentOnBdry}.
%
 \end{proof}

 \begin{figure}
   \centering
   \includegraphics[width=0.45\linewidth]{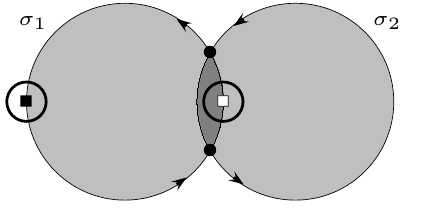}
   \caption{
    In proving Lemma \ref{lem:bryEdgeRequirement}, we 
     consider ${\cal Y}:=\mathrm{int}(\sigma_1)\cap \mathrm{int}(\sigma_2)$
     where $\sigma_1$ and $\sigma_2$ are two oriented Jordan curves.
    By Corollary \ref{coro:brdPtOfRegularSets},
     $\beta\subset \partial {\cal Y}$ if and only if
     for any point $p\in \beta$ (the open square),
     any sufficiently small neighborhood of $p$
     contains both points in ${\cal Y}$ and ${\cal Y}^{\perp}$.
    On the other hand, 
     if a point $q$ (the filled square)
     is not in $\mathrm{\sigma_2}$,
     then it is definitely not on the boundary of ${\cal Y}$.
   }
 \label{fig:segmentOnBdry}
\end{figure}


Hereafter we write
 the union of all Jordan curves in a realizable spadjor as
\begin{equation}
  \label{eq:pointSetOfSpadjorForest}
  P({\cal J}) := \bigcup_{\gamma_i\in {\cal J}} \gamma_i, 
\end{equation}
 which is clearly a subset of $\mathbb{R}^2$.

 \begin{corollary}
   \label{coro:bryEdgeRequirementForest}
   Denote 
   ${\cal Y}:=\rho({\cal J})\cap \rho({\cal K})$
   where ${\cal J}$ and ${\cal K}$ are two realizable spadjors.
   For a curve $\beta$ satisfying 
    $\beta\subseteq P({\cal J})$
    and $\beta\subset \mathbb{R}^2\setminus P({\cal K})$,
    we have $\beta\subseteq \partial {\cal Y}$ 
    if and only if $\beta\subset \rho({\cal K})$.
   For a curve $\beta$ satisfying 
     $\beta\subseteq P({\cal J})\cap P({\cal K})$,
     we have $\beta\subseteq \partial {\cal Y}$ 
     if and only if 
     the direction of $\beta$
     induced from the orientation of ${\cal J}$
     is the same as that
     induced from the orientation of ${\cal K}$.
   \end{corollary}
 

 \begin{definition}
   \label{def:MRS}
   The \emph{meet of two realizable spadjors}
    ${\cal J}$ and ${\cal K}$
    is a binary operation
    $\wedge: \mathbb{J}\times\mathbb{J}\rightarrow\mathbb{J}$
    defined as
    \begin{equation}
      \label{eq:MRS}
      {\cal J} \wedge {\cal K} = \left\{
        \begin{array}{ll}
          \Zero  & \textrm{if } {\cal K}=\Zero;
          \\
          {\cal J}  & \textrm{if } {\cal K}=\One;
          \\
          S_V^{-1}(E) & \textrm{otherwise}, 
        \end{array}
        \right.
    \end{equation}
    where the pasting map $S_V^{-1}$
    is defined in Lemma \ref{lem:segmentationMapInverse},
    and the directed multigraph $(V,E)$ is constructed as follows.
   \begin{enumerate}[(MRS-a)]
   \item The set ${\cal I}:=P({\cal J}) \cap P({\cal K})$
     may contain paths and isolated points.
    Initialize $V$ as an empty set;
     add into $V$ all path endpoints and isolated points 
     in ${\cal I}$.
  \item Cut ${\cal J}$ with points in $V$
    and we have a set of paths
    $\{\beta_i\}=S_V({\cal J})$.
    Initialize $E$ as an empty set. 
  \item For each $\beta_i$,
    add it to $E$ if $\beta_i$ minus its endpoints is contained
    in $\rho({\cal K})$,
    or, if there exists $\beta_j\subset P({\cal K})$
    such that $\beta_j=\beta_i$ and they have the same direction.
    In particular, if $\beta_i$ is a Jordan curve
    that satisfies either of the above conditions,
    we insert $\beta_i$ as a self-loop into $E$.
  \item For each $\beta_j\subset S_V({\cal K})$,
    add it to $E$ if $\beta_j$ minus its endpoints
    is contained in $\rho({\cal J})$.
  \end{enumerate}
  \end{definition}

  \begin{lemma}
    \label{lem:meetHomo}
    The meet operation in Definition \ref{def:MRS}
     satisfies 
     \begin{equation}
       \label{eq:meetHomo}
       \forall {\cal J}, {\cal K}\in \mathbb{J},
       \qquad \rho({\cal J} \wedge {\cal K})
       = \rho({\cal J}) \cap \rho({\cal K}).
     \end{equation}
  \end{lemma}
  \begin{proof}
    Denote ${\cal Y}_{\cap}:=\rho({\cal J}) \cap \rho({\cal K})$.
    It follows from Corollary \ref{coro:bryEdgeRequirementForest}
     that each path $\beta_i\subset S_V({\cal J})$
     is added to $E$ in step (MRS-c) 
     if and only if 
     $\beta_i\subset \partial {\cal Y}_{\cap}$;
     similarly, 
     each curve $\beta_j\subset S_V({\cal K})$
     is added to $E$ in step (MRS-d) 
     if and only if 
     $\beta_j\subset \partial {\cal Y}_{\cap}$.
    Hence the union of the vertices and edges
     in $G$ constitute the boundary of ${\cal Y}_{\cap}$.
    Furthermore, by the difference between (MRS-c) and (MRS-d), 
     each edge on $\partial{\cal Y}_{\cap}$ is inserted into $E$ only once.
    Therefore,
     $E$ contains and only contains
     points on $\partial{\cal Y}_{\cap}$.
    The proof is then completed by Lemma
    \ref{lem:segmentationMapInverse}
    and Corollary \ref{coro:uniqueRepOfYinSets}.
  \end{proof}

\subsection{The Yin space $\mathbb{Y}$
 and the Jordan space $\mathbb{J}$
 are isomorphic}
\label{sec:isomorphism}

The join operation can be expressed by the meet operation
 and the complementation operation;
 this is also true for all other Boolean operations.
 
 \begin{definition}
   \label{def:join}
   The \emph{join of two realizable spadjors} ${\cal J}$ and ${\cal K}$
    is a binary operation
    $\vee: \mathbb{J}\times\mathbb{J}\rightarrow\mathbb{J}$
    defined as
    \begin{equation}
      \label{eq:join}
      \forall {\cal J}, {\cal K}\in \mathbb{J},\qquad
      {\cal J} \vee {\cal K}
      := ({\cal J}' \wedge {\cal K}')'.
    \end{equation}
  \end{definition}

The following theorem is the theoretical culmination of this paper.

\begin{theorem}
  \label{thm:isomorphism}
  The Boolean algebras $(\mathbb{J},\vee,\wedge,',\Zero,\One)$
  and
  $(\mathbb{Y},\cup^{\perp\perp},\cap,^{\perp},\emptyset,\mathbb{R}^2)$
  are isomorphic 
  under the boundary-to-interior map
  $\rho$ in Definition \ref{def:boundaryToInteriorMap}.
\end{theorem}
\begin{proof}
  This follows directly from 
  Lemmas \ref{lem:complementationHomo} and \ref{lem:meetHomo},
  Definitions \ref{def:boundaryToInteriorMap} and \ref{def:join},
  and DeMorgan's law (\ref{eq:DeMorgansLaws}).
\end{proof}

As the desired consequence,
 we have reduced the two-dimensional problems
 $\cup^{\perp\perp}$, $\cap$, and $^{\perp}$
 to the one-dimensional problems
 $\vee$, $\wedge$, and $'$.

\subsection{Algorithmic implementation}
\label{sec:algor-impl}

 \begin{figure}
   \centering
   \subfigure[${\cal P}^{\perp}$:
   exterior of the panda 
   ${\cal P}$
   in Figure \ref{fig:panda}
   obtained by complementation
   in Definition \ref{def:CJJ}.
   ]{
     \includegraphics[width=0.45\linewidth]{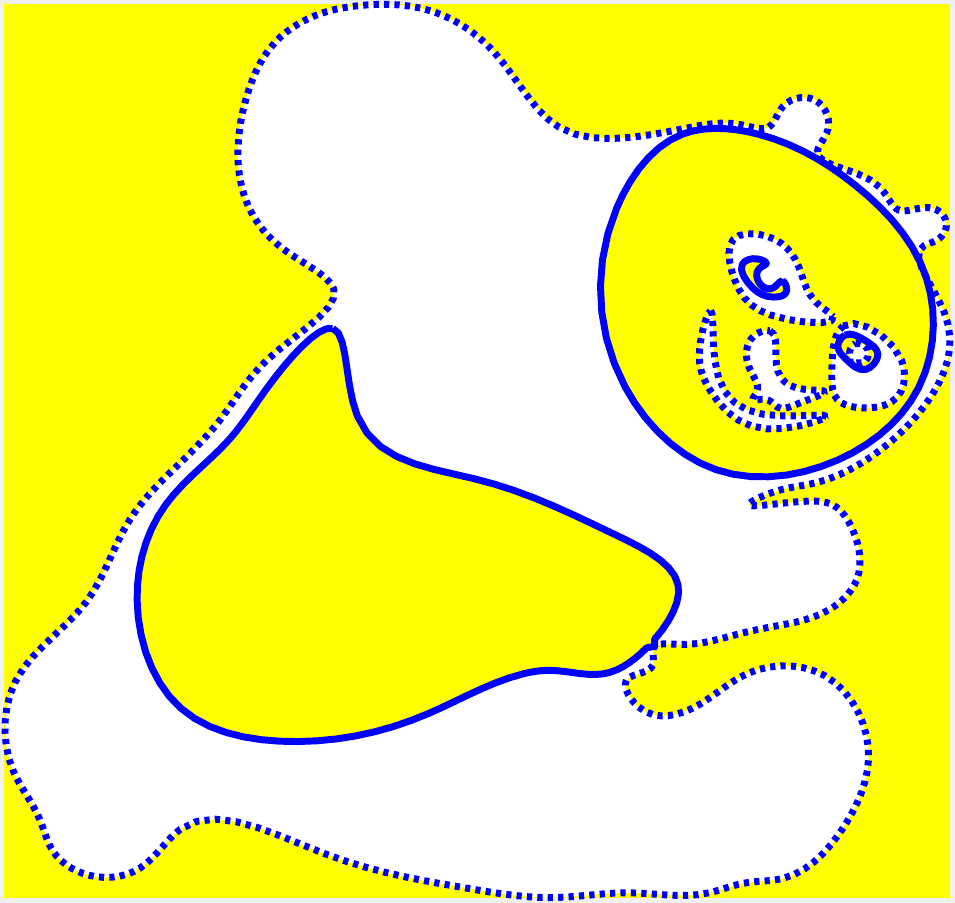}
   }
   \subfigure[a Mickey mouse modeled as a Yin set ${\cal M}$]{
     \includegraphics[width=0.45\linewidth]{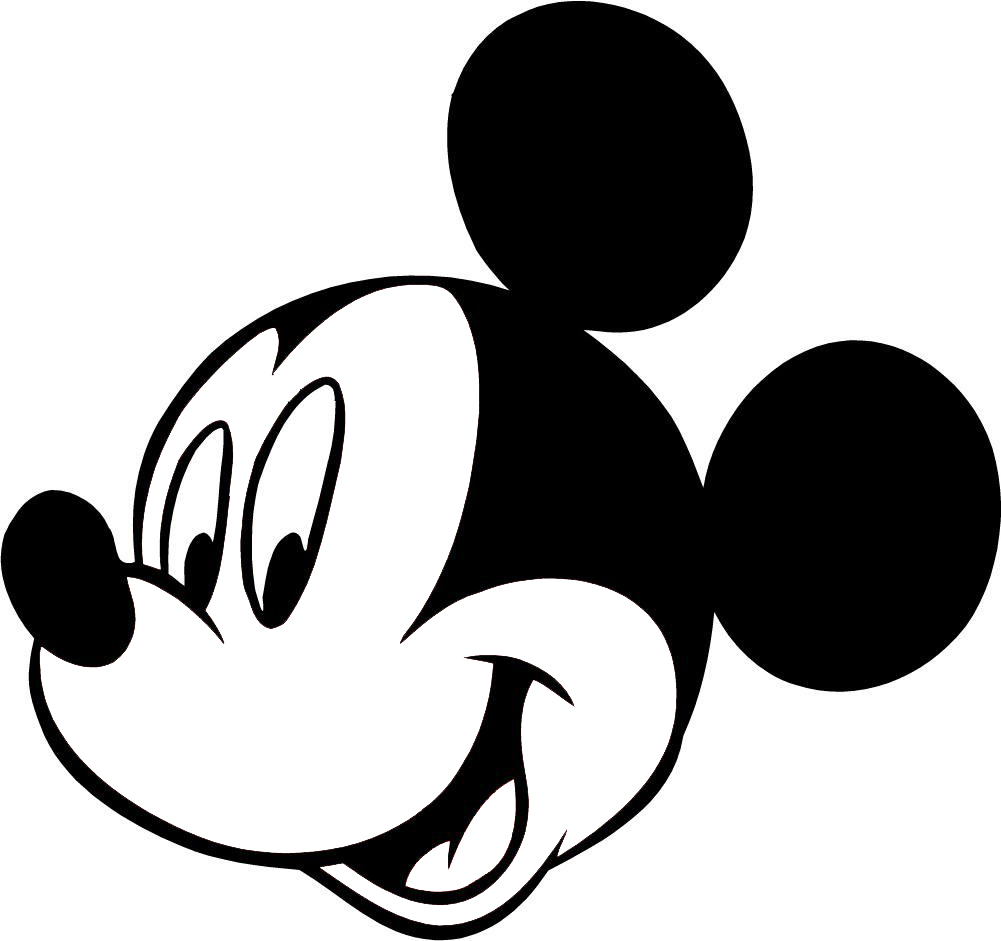}
   }

   \subfigure[${\cal M}\cap {\cal P}$:
   intersection of ${\cal M}$ and ${\cal P}$
   obtained by the meet operation in Definition \ref{def:MRS}.
   ]{
     \includegraphics[width=0.46\linewidth]{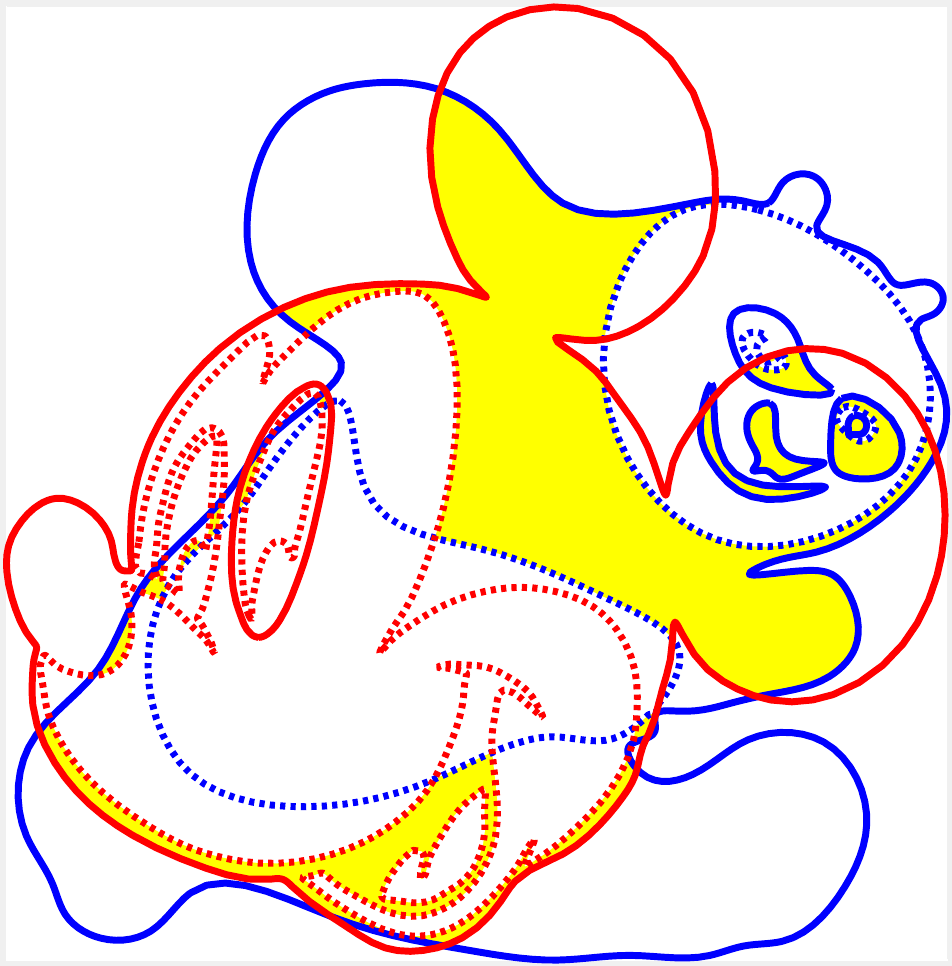}
   }
   \hfill
   \subfigure[${\cal M}\cup^{\perp\perp}{\cal P}$:
   regularized union of ${\cal M}$ and ${\cal P}$
   obtained by the join operation in Definition \ref{def:join}.
   ]{
     \includegraphics[width=0.46\linewidth]{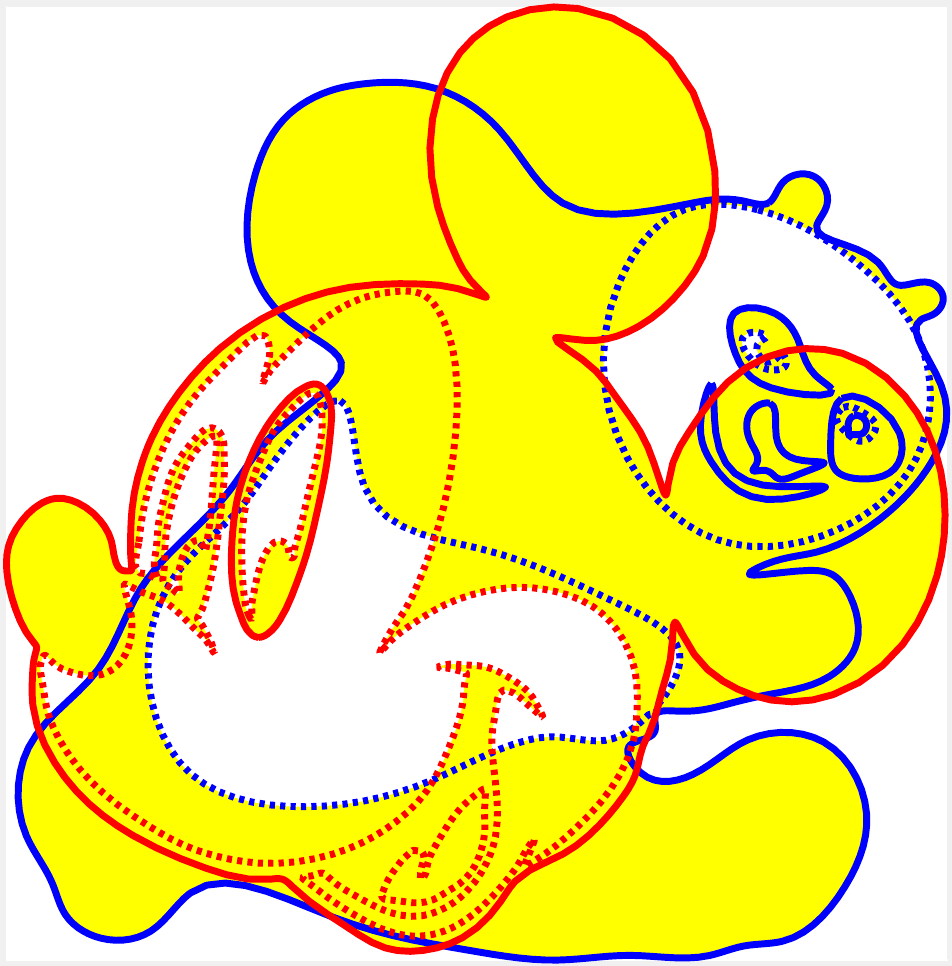}
   }
   \caption{Results of testing Boolean algorithms
      on Yin sets with complex topology and geometry.
     In subplots (a), (c), and (d),
      a solid line represents a positively
      oriented Jordan curve,
      a dotted line a negatively
      oriented Jordan curve,
      and a shaded region the result of a Boolean operation.
     There are several improper intersections
      of the Jordan curves in the realizable spadjor
      that represents the panda.
   }
 \label{fig:panda-mickey}
\end{figure}

Purely algebraic and constructive as they are,
 Definitions 
 \ref{def:CJJ}, \ref{def:MRS}, and \ref{def:join}
 already constitute a complete set of Boolean algorithms on Yin sets.
We implement these algorithms
 and perform a variety of test cases
 to validate our theory and verify our implementation;
 some fun examples are shown in Figure \ref{fig:panda-mickey}.

In our implementation, 
 the data structure of Yin sets is
 a straightforward orchestration of the realizable spadjor,
 i.e. a set of point arrays
 and each point array represent a polygon,
 with its orientation indicated by the direction of the points.
In particular,
 the user does not need to specify the pairwise inclusion relations
 of the boundary Jordan curves,
 since the computation of this information is
 encapsulated inside the algorithms.
So long as each input Yin set
 is indeed a realizable spadjor,
 the algorithm returns the correct result.
These designs make the software interface 
 simple, flexible, and user-friendly.

As one distinguishing feature,  
 the user can control the uncertainty 
 of Boolean operations of our implementation.
\emph{Given a small positive real number $\epsilon$,
 we define two points to be the same point
 if their distance is smaller than $\epsilon$}.
A consistent enforcement of this definition and its implications
 across the entire package
 makes our implementation robust
 and provides an effective mechanism to handle various degenerate cases
 that characterize topological changes.

It is well known in the community of computational geometry
 that intersecting line segments
 might lead to unavoidable self-inconsistencies
 and cause a program to abort at the run time
 \cite{kettner08:_class_examp_of_robus_probl}.
The mathematical core of this difficulty
 is the potentially arbitrary ill-conditioning
 of intersecting line segments
 in an unlimited range of length scales.
In the context of numerically simulating multiphase flows,
 there always exists a length scale $h$
 below which finer details are not needed.
Hence this uncertain parameter $\epsilon$
 is not only a device of flexibility and convenience, 
 but more importantly a simple fix
 of the aforementioned robustness problem in computational geometry.

In calculating intersections of a set of line segments,
 we modify the standard line sweep algorithm
 in \cite{bentley79:_algor}\cite[Chap. 2]{berg08:_comput_geomet}
 to enforce this uncertainty criterion.
We also need to frequently determine whether or not
 a point belongs to the interior, the exterior,
 or the boundary of a polygon;
 for this purpose,
 we expand the simple algorithm
 in \cite[Sec. 7.4]{orourke98:_comput_geomet_in_c}
 so that it applies not only to simple polygons
 but also to realizable spadjors.
These algorithmic details of our implementation
 will be reported in a separate paper.



\section{Conclusion}
\label{sec:conclusion}

We have introduced the problem of fluid modeling
 in multiphase flows as a counterpart of solid modeling
 in CAD,
 and have proposed to solve this problem
 via the Yin space,
 a topological space equipped
 with a simple, efficient, and complete Boolean algebra.
Under this framework,
 topological changes of deforming Yin sets
 can be characterized and handled naturally
 and topological information such as Betti numbers
 can be extracted in constant time.

Several prospects for future research follow.
The theory and algorithms on the Yin space
 can be generalized to 2-manifolds in a straightforward manner.
Another generalization of Yin sets
 to three dimensions is currently a work in progress.
Finally, it would be exciting to couple Yin sets
 with high-order finite-volume methods \cite{zhang16:_GePUP}
 to form fourth- and higher-order solvers
 for simulating multiphase flows
 such as free-surface flows
 and fluid-structure interactions.



{\bf Acknowledgments.}
This work was supported by the grant with approval number 11871429
 from the national natural science foundation of China.
The authors thank Difei Hu for digitizing the panda image.


\bibliography{bib/YinSets2D}
\bibliographystyle{siam}

\end{document}